\documentclass[letterpaper, 10 pt, conference]{ieeeconf}  

\IEEEoverridecommandlockouts                              
\usepackage{graphicx}
\usepackage{amsfonts}
\usepackage{array}
\usepackage{amssymb}
\usepackage{amsmath}
\usepackage{algorithm}
\usepackage{caption}
\usepackage{subcaption}
\usepackage{algpseudocode}
\usepackage{bm}
\usepackage{color}

\newtheorem{thm}{Theorem}[section]

\newtheorem{defn}[thm]{Definition}
\newtheorem{cor}{Corollary}
\newtheorem{rem}{Remark}[section]

\newcommand{\Exp}{\mathbb{E}}

\newcommand{\R}{\mathcal{R}}

\newcommand{\cG}{\mathcal{G}} 
\newcommand{\cV}{\mathcal{V}} 
\newcommand{\cE}{\mathcal{E}} 
\newcommand{\cL}{\mathcal{L}} 
\newcommand{\cS}{\mathcal{S}} 

\newcommand{\bA}{\mathbf{A}}
\newcommand{\bB}{\mathbf{B}}

\newcommand{\bH}{\mathbf{H}}
\newcommand{\bI}{\mathbf{I}}

\newcommand{\bS}{\mathbf{S}}

\newcommand{\bW}{\mathbf{W}}

\newcommand{\eqdef}{:=}

\newcommand{\cD}{{\cal D}}

\newcommand{\mA}{{\bf A}}
\newcommand{\mB}{{\bf B}}

\newcommand{\mH}{{\bf H}}

\newcommand{\mL}{{\bf L}}

\newcommand{\mS}{{\bf S}}

\usepackage[colorinlistoftodos,bordercolor=orange,backgroundcolor=orange!20,linecolor=orange,textsize=scriptsize]{todonotes}

\title{\LARGE \bf
Accelerated Gossip via Stochastic Heavy Ball Method\thanks{Accepted for publication to 56th Annual Allerton Conference on Communication, Control, and Computing. This work appeared first time online on 9th July 2018.}
}

\author{ \parbox{3 in}{\centering Nicolas Loizou\\ 
         School of Mathematics,\\ The University of Edinburgh\\
        Edinburgh, Scotland, UK\\
         {\tt\small n.loizou@sms.ed.ac.uk}}
         \hspace*{ 0.5 in}
         \parbox{3 in}{ \centering Peter Richt\'{a}rik\\
        KAUST, Saudi Arabia \\
        The University of Edinburgh, United Kingdom\\
        MIPT, Russia\\
         {\tt\small peter.richtarik@kaust.edu.sa}}
}
\date{}

\begin{document}

\maketitle
\thispagestyle{empty}
\pagestyle{empty}

\begin{abstract}
In this paper we show how the stochastic heavy ball method (SHB)---a  popular method for solving stochastic convex and non-convex optimization problems---operates as a randomized gossip algorithm. 
 In particular, we focus on two special cases of SHB: the Randomized Kaczmarz method with momentum and its block variant. Building upon a recent framework  for the design and analysis of randomized gossip algorithms  \cite{LoizouRichtarik} we interpret the distributed nature of the proposed methods. We present novel protocols for solving the average consensus problem where in each step all nodes of the network update their values but only a subset of them exchange their private values. Numerical experiments on popular wireless sensor networks showing the benefits of our protocols are also presented.
\end{abstract}

\section{Introduction}

Average consensus is a fundamental problem in distributed computing and multi-agent systems. It comes up in many real world applications such as coordination of autonomous agents, estimation, rumour spreading in social networks, PageRank and distributed data fusion on ad-hoc networks and decentralized optimization. Due to its great importance there is much classical \cite{tsitsiklis1986distributed,degroot1974reaching} and recent   \cite{xiao2005scheme, xiao2004fast, boyd2006randomized} work on the design of efficient algorithms/protocols for solving it.

One of the most attractive classes of protocols for solving the average consensus are gossip algorithms. The development and design of gossip algorithms was studied extensively in the last decade. The seminal 2006 paper of Boyd et al.\ \cite{boyd2006randomized} on randomized gossip algorithms motivated  a fury of subsequent research and now gossip algorithms appear in many applications, including distributed data fusion in sensor networks \cite{xiao2005scheme}, load balancing \cite{cybenko1989dynamic} and clock synchronization \cite{freris2012fast}.  For a survey of selected relevant work prior to 2010, we refer the reader to the work of  Dimakis et al.\ \cite{dimakis2010gossip}. For more recent results on randomized gossip algorithms we suggest \cite{zouzias2015randomized, liu2013analysis,olshevsky2014linear,LoizouRichtarik, nedic2018network, aybat2017decentralized}. 
See also  \cite{dimakis2008geographic, aysal2009broadcast, olshevsky2009convergence,hanzely2017privacy}. 

The main goal in the design of gossip protocols is for the computation and communication to be done as quickly and efficiently as possible. In this work, our focus is precisely this. We design randomized gossip protocols which converge to consensus fast.

\subsection{The average consensus problem} \label{sec:ACP}

In the average consensus (AC) problem we are given an undirected connected network $\cG=(\cV,\cE)$ with node set $\cV=\{1,2,\dots,n\}$ and edges $\cE$. Each node $i \in \cV$ ``knows'' a private value $c_i \in \R$. The goal of AC is for every node to compute the average of these private values, $\bar{c}\eqdef\tfrac{1}{n}\sum_i c_i$, in a distributed fashion. That is, the exchange of information can only occur between connected nodes (neighbors). 

\subsection{Main Contributions} 

We present a new class of randomized gossip protocols where in each iteration all nodes of the network update their values  but only a subset of them exchange their private information. Our protocols  are based on recently proposed ideas for the acceleration of randomized Kaczmarz methods for solving consistent linear systems \cite{loizou2017momentum} where the addition of a momentum term was shown to provide practical speedups over the vanilla Kaczmarz methods.  Further, we explain the connection between  gossip algorithms for solving the average consensus problem, Kaczmarz-type methods for solving consistent linear systems, and stochastic gradient descent and stochastic heavy ball methods for solving stochastic optimization problems. We show that essentially all these algorithms behave as gossip algorithms.  Finally, we explain in detail the gossip nature of two recently proposed fast Kacmzarz-type methods:  the randomized Kacmzarz with momentum (mRK), and its block variant, the randomized block Kaczmarz with momentum (mRBK). We present a detailed comparison of our proposed gossip protocols with existing popular randomized gossip protocols and through numerical experiments we show the benefits of our methods.

\subsection{Structure of the paper}

This work is organized as follows. Section~\ref{background} introduces the important technical preliminaries and the necessary background for understanding of our methods. A new connection between gossip algorithms, Kaczmarz methods for solving linear systems and stochastic gradient descent (SGD) for solving stochastic optimization problems is also described.  In Section~\ref{theProtocols} the two new accelerated gossip protocols are presented. Details of  their behaviour and performance are also explained. Numerical evaluation of the new gossip protocols is presented in Section~\ref{experiments}. 
Finally, concluding remarks are given in Section~\ref{conclusion}.

\subsection{Notation}
The following notational conventions are used in this paper. We write $[n]\eqdef \{1,2, \dots ,n\}$.  Boldface upper-case letters denote matrices; $\bI$ is the identity matrix. 
By $\cL$ we denote the solution set of the linear system $\bA x=b$, where $\bA \in \R^{m\times n}$ and $b\in \R^m$. Throughout the paper, $x^*$ is the projection of $x^0$ onto $\cL$ (that is, $x^*$  is the solution of the best approximation problem; see equation \eqref{best approximation}). 
An explicit formula for the projection of $x$ onto set $\cL$ is given by
\begin{equation*}
\Pi_{\cL}(x)\eqdef \arg\min_{x' \in \cL} \|x'-x\| =x-\bA^\top (\bA\bA ^ \top )^\dagger (\bA x-b).
\end{equation*}
A matrix that  often appears in our update rules is 
\begin{equation}
\label{MatrixH}
\mH \eqdef  \mS (\mS^\top \mA\mA^\top \mS)^\dagger \mS^\top,
\end{equation}
where $\mS \in \R^{m \times q}$ is a  random matrix drawn in each step of the proposed methods from a given distribution $\cD$, and   $\dagger$ denotes the Moore-Penrose pseudoinverse.  Note that 
 $\mH$ is a random symmetric positive semi-definite matrix.
 
 In the convergence analysis we use $\lambda_{\min}^+$ to indicate the smallest nonzero eigenvalue, and $\lambda_{\max}$ for the largest eigenvalue of matrix $\bW=\Exp[\bA^\top \bH \bA]$, where the expectation is taken over $\bS\sim \cD$.   Finally, $x^k = (x^k_1,\dots,x^k_n) \in \R^n$ represents the vector with the private values of the $n$ nodes of the network at the $k^{th}$ iteration while with $x_i^{k}$ we denote the value of node $i \in [n]$ at the $k^{th}$ iteration.

\section{Background-Technical Preliminaries}
\label{background}
Our work is closely related to two recent papers. In \cite{LoizouRichtarik},  a new perspective on randomized gossip algorithms is presented. In particular, a new approach for the design and analysis of randomized gossip algorithms is proposed and it was shown how the Randomized Kaczmarz and Randomized Block Kaczmarz, popular methods for solving linear systems, work as gossip algorithms when applied to a special system encoding the underlying network. In \cite{loizou2017momentum}, several classes of stochastic optimization algorithms enriched with  {\em heavy ball momentum} were analyzed. Among the methods studied are: stochastic gradient descent, stochastic Newton, stochastic proximal point  and stochastic dual subspace ascent.

In the rest of this section we present the main results of the above papers, highlighting several connections. These results will be later used for the development of the new randomized gossip protocols.


\subsection{Kaczmarz Methods and Gossip Algorithms}

Kaczmarz-type methods are very popular  for solving linear systems $\bA x =b$ with many equations. The (deterministic) Kaczmarz method for solving consistent linear systems was originally introduced by Kaczmarz in 1937 \cite{1937angenaherte}. Despite the fact that a large volume of papers  was written  on the topic, the first provably linearly convergent variant of the Kaczmarz method---the randomized Kaczmarz Method (RK)---was developed more than 70 years later, by Strohmer and Vershynin \cite{RK}. This result sparked renewed interest in design of randomized methods for solving linear systems \cite{needell2010randomized, RBK, eldar2011acceleration, MaConvergence15, zouzias2013randomized, l2015randomized, schopfer2016linear, liu2016accelerated}. More recently, Gower and Richt\'{a}rik \cite{gower2015randomized} provide a unified analysis for several randomized iterative methods for solving linear systems using a sketch-and-project framework. We adopt this framework in this paper. 

In particular, the sketch-and-project algorithm  \cite{gower2015randomized} for solving the consistent linear system $\bA x= b$ has the form
\begin{eqnarray}
\label{sketchproject}
x^{k+1} &=& x^k -\bA^\top \bS_k (\bS_k^\top \bA \bA^\top \bS_k)^\dagger \bS_k^\top (\bA x^k-b) \notag\\ 
&\overset{\eqref{MatrixH}}{=}& x^k -\bA^\top \bH_k (\bA x^k-b),
\end{eqnarray}
where in each iteration matrix $\mS_k$ is sampled afresh from an arbitrary distribution $\cD$. In \cite{gower2015randomized} it was shown that many popular algorithms for solving linear systems, including RK method and randomized coordinate descent method can be cast as special cases of the above update by choosing\footnote{In order to recover a randomized coordinate descent method, one also needs to perform projections with respect to a more general Euclidean norm. However, for simplicity, in this work we only consider the standard Euclidean norm.} an appropriate distribution $\cD$. The special cases that we are interested in are the
randomized Kaczmarz (RK) and its block variant, the randomized block Kaczmarz (RBK).

Let  $e_i \in \R^m$  be the $i^{\text{th}}$ unit coordinate vector in $ \R^m$ and let $\bI_{:C}$ be column submatrix of the $m \times m$ identity matrix with columns indexed by  $C\subseteq [m]$.  Then RK and RBK methods can be obtained as special cases of the  update rule \eqref{sketchproject} as follows:
\begin{itemize}
\item RK: Let $\bS_k=e_i$, where $i=i_k$ is chosen in each iteration independently, with probability $p_i>0$. In this setup the update rule \eqref{sketchproject} simplifies to  
\begin{equation}
\label{RK}
x^{k+1}=x^k - \tfrac{\bA_{i :} x^k -b_{i}}{\|\bA_{i :}\|_2^2} \bA_{i :}^ \top  .
\end{equation}
\item RBK: Let $\bS=\bI_{:C}$, where $C=C_k$ is chosen in each iteration independently, with probability $p_C\geq 0$. In this setup the update rule \eqref{sketchproject} simplifies to 
\begin{equation}
\label{RBK}
x^{k+1}=x^k - \bA_{C:}^\top (\bA_{C:}\bA_{C:}^\top)^\dagger (\bA_{C:}x^k-b_C).
\end{equation}
\end{itemize}

In this paper we are interested in two particular extension of the above methods: the randomized Kaczmarz method with momentum (mRK) and its block variant, the randomized block Kaczmarz with momentum (mRBK), both proposed and analyzed in \cite{loizou2017momentum}. Before we describe these two algorithms, let us summarize the main connections between the  Kaczmarz methods for solving linear systems and gossip algorithms, as presented in \cite{LoizouRichtarik}.

In \cite{gower2015stochastic, ASDA, loizou2017momentum}, it was shown that even in the case of consistent linear systems with {\em multiple} solutions, Kaczmarz-type methods converge linearly to one particular solution: the projection of the initial iterate $x^0$ onto the solution set of the linear system. This naturally leads to the formulation of the {\em best approximation problem}:
\begin{equation}
\label{best approximation}
\min_{x = (x_1,\dots, x_n) \in \R^n} \tfrac{1}{2} \|x-x^0\|^2 
\quad \text{subject to}  \quad \bA x = b.
\end{equation}
Above, $\bA\in \R^{m\times n}$ and $\|\cdot\|$ is the standard Euclidean norm.  By $x^*=\Pi_{\cL}(x^0)$ we denote the  solution of \eqref{best approximation}.

In \cite{LoizouRichtarik} it was shown how RK and RBK work as gossip algorithms when applied to a special linear system encoding the underlying network.
\begin{defn}[\cite{LoizouRichtarik}] A linear system  $\bA x = b$ is called ``average consensus (AC) system'' when  $\bA x = b$ is equivalent to saying that $x_i = x_j$ for all $(i,j) \in \cE$.
\end{defn}

Note that many linear systems satisfy the above definition. For example, we can choose $b=0$ and $\bA \in \R^{|\cE| \times n}$ to be the incidence matrix of  $\cG$. In this case, the row of the system  corresponding to edge $(i,j)$ directly encodes the constraint $x_i=x_j$. 
A different choice is to pick $b=0$ and  $\mA =\mL$, where $\mL$ is the Laplacian of $\cG$. Note that depending on what AC system is used, RK and RBK have different interpretations as gossip protocols. 

From now on  we work with the AC system described in the first example. Since $b=0$, the general sketch-and-project update rule \eqref{sketchproject} simplifies to:
\begin{equation}
\label{updateSkProj}
x^{k+1}= \left[ \bI- \bA^\top \bH_k\bA \right] x^k.
\end{equation}

The convergence performance of RK  and RBK for solving the best approximation problem (and as a result the average consensus problem) is described by the following theorem.
\begin{thm}[\cite{gower2015randomized,gower2015stochastic}]
 Let $\{x^k\}$ be the iterates produced by \eqref{sketchproject}. Then
$\Exp[\|x^k-x^*\|^2]\leq \rho^k \|x^0-x^*\|^2,$
where $x^*$ is the solution of \eqref{best approximation},
$\rho \eqdef 1 - \lambda_{\min}^+ \in [0,1]$, and 
$\lambda_{\min}^+$ denotes the minimum nonzero eigenvalue of $\bW\eqdef \Exp[\bA^\top \bH \bA]$. 
\end{thm}

In \cite{LoizouRichtarik}, the behavior of both RK and RBK as gossip algorithms was described, and a comparison with the convergence results of existing randomized gossip protocols was made. In particular, it was shown that the most basic randomized gossip algorithm \cite{boyd2006randomized}  (``randomly pick an edge $(i,j)\in \cE$ and then replace the values stored at vertices $i$ and $j$ by their average'') is an instance of RK   applied to the linear system $\bA x=0 $, where the $\bA$ is the incidence matrix of  $\cG$. RBK can also be  interpreted  as a gossip algorithm:
\begin{thm}[\cite{LoizouRichtarik}, RBK as a Gossip Algorithm]
\label{TheoremRBK}
Each iteration of RBK for solving $\bA x=0$ works as follows:
1) Select a random set of edges $\cS \subseteq \cE$, 
2) Form subgraph $\cG_k$ of $\cG$ from the selected  edges, 
3) For each connected component of $\cG_k$, replace node values with their average.
\end{thm}


\subsection{The Heavy Ball momentum}
A detailed study of several (equivalent) {\em stochastic reformulations} of consistent linear systems was developed in \cite{ASDA}. This new viewpoint facilitated the development and   analysis of relaxed variants (with relaxation parameter $\omega \in (0,2)$) of the sketch-and-project update \eqref{sketchproject}. In particular, one of the reformulations is the {\em stochastic optimization} problem
\begin{equation}
\label{stoch_reform}
\min_{x\in \R^n} f(x) \eqdef \Exp_{\mS\sim \cD}[f_\mS(x)], \quad \text{where}
\end{equation}
\begin{equation}
\label{eq:f_s}
f_{\mS}(x) \eqdef \tfrac{1}{2}\|\mA x - b\|_{\mH}^2 = \tfrac{1}{2}(\mA x - b)^\top \mH (\mA x - b),
\end{equation}
and $\mH$ is the random symmetric positive semi-definite matrix defined in \eqref{MatrixH}.

Under certain (weak) condition on $\cD$,  the set of minimizers of  $f$ is identical to the set of the solutions of the  linear system. In \cite{ASDA}, problem \eqref{stoch_reform} was solved via Stochastic Gradient Descent (SGD): 
\begin{equation}
\label{SGD}
x^{k+1}=x^k-\omega \nabla f_{\bS_k}(x^k),
\end{equation} 
and a linear rate of convergence was proved despite the fact that  $f$ is not necessarily strongly convex and that a fixed stepsize $\omega>0$ is used. Observe that the gradient of the stochastic function \eqref{eq:f_s} is given by
\begin{equation}
\label{stochGradi}
\nabla f_{\mS_k} (x) \overset{\eqref{eq:f_s}}{=}  \mA^\top \mH_k (\mA x - b).
\end{equation}
and as a result, it is easy to see that for $\omega=1$, the SGD update \eqref{SGD} reduces to the  sketch-and-project update~\eqref{sketchproject}.

The recent works \cite{loizou2017linearly,loizou2017momentum} analyze momentum variants of SGD, with the goal to accelerate the convergence of the method for solving problem \eqref{stoch_reform}. SGD with momentum---also known as the stochastic heavy ball method (SHB)---is a well known algorithm in the optimization literature for solving stochastic optimization problems, and it is extremely popular in areas such as deep learning \cite{sutskever2013importance, szegedy2015going, krizhevsky2012imagenet, wilson2017marginal}. However, even though SHB is used extensively in practice, its theoretical convergence behavior is not well understood. To the best of our knowledge,  \cite{loizou2017linearly,loizou2017momentum} are the first that prove linear convergence of SHB in any setting.

The update rule of SHB for solving problem \eqref{stoch_reform} is formally presented in the following algorithm:
\begin{algorithm}[H]
	\caption{Stochastic Heavy Ball (SHB) }
	\label{alg:SARAHHB}
	\small \small
	\begin{algorithmic}[1]
		\State {\bf Parameters:} Distribution $\mathcal{D}$ from which method samples matrices; stepsize/relaxation parameter $\omega \in \R$; momentum parameter $\beta$.
		\State {\bf Initialize:} $x^0,x^1 \in \R^n$
		\For{$k=1,2,\dots$} 
		\State Draw a fresh $\bS_k \sim \cD$
		\State   Set 
$x^{k+1}=x^k-\omega \nabla f_{\bS_k}(x^k)+\beta (x^k-x^{k-1})
$
		\EndFor
		\State {\bf Output:} The last iterate $x^k$
	\end{algorithmic}
\end{algorithm}

Using the expression for the stochastic gradient \eqref{stochGradi}, the update rule of SHB can  be written more explicitly:
\begin{equation}
\label{SPmomentum}
x^{k+1}=x^k -\omega \bA^\top \bH_k(\bA x^k-b) + \beta(x^k - x^{k-1}).
\end{equation}
Using the same choice of distribution $\cD$ as in equation \eqref{RK} and \eqref{RBK}, we now obtain momentum variants of RK and RBK: 
\begin{itemize}
\item RK with momentum (mRK): $$x^{k+1}=x^k -\omega \tfrac{\bA_{i :} x^k -b_{i}}{\|\bA_{i :}\|_2^2} \bA_{i :}^ \top + \beta(x^k - x^{k-1}) $$
\item RBK with momentum (mRBK):
$$x^{k+1}=x^k -\omega \bA_{C:}^\top (\bA_{C:}\bA_{C:}^\top)^\dagger (\bA_{C:}x^k-b_C)+ \beta(x^k - x^{k-1}) $$
\end{itemize}

In \cite{loizou2017momentum}, two main theoretical results  describing the behavior of SHB (and as a result also the special cases mRK and mRBK) were presented:
\begin{thm}[Theorem 1, \cite{loizou2017momentum}]
\label{L2}
Choose $x^0= x^1\in \R^n$. Let $\{x^k\}_{k=0}^\infty$ be the sequence of random iterates produced by SHB.  Let $\lambda_{\min}^+$ (resp,\ $\lambda_{\max}$) be the smallest nonzero (resp.\ largest) eigenvalue of $\bW$. Assume $0< \omega < 2$ and $\beta \geq 0$ and that the expressions
$a_1 \eqdef 1+3\beta+2\beta^2 - (\omega(2-\omega) +\omega\beta)\lambda_{\min}^+$ and
$a_2 \eqdef \beta +2\beta^2 + \omega \beta \lambda_{\max}$
satisfy $a_1+a_2<1$. Then 
\begin{equation}\label{eq:nfiug582}\Exp[\|x^{k}-x^*\|^2] \leq q^k (1+\delta)  \|x^0-x^*\|^2,
\end{equation}
and 
$\Exp[f(x^k)] \leq q^k  \tfrac{\lambda_{\max}}{2} (1+\delta) \|x^{0}-x^*\|^2,$
where  $q=\tfrac{1}{2} (a_1+\sqrt{a_1^2+4a_2})$ and $\delta=q-a_1$. Moreover, $a_1+a_2 \leq q <1$.
\end{thm}

\begin{thm}[Theorem 4, \cite{loizou2017momentum}]
\label{theoremheavyball}
Let $\{x^k\}_{k=0}^{\infty}$ be the sequence of random iterates produced by SHB, started with $x^0= x^1\in \R^n$, with relaxation parameter (stepsize)  $0<\omega \leq1/\lambda_{\max}$ and momentum parameter  $(1-\sqrt{\omega \lambda_{\min}^+})^2 < \beta <1$. Let $x^* = \Pi^\mB_{\cL}(x^0)$. Then there exists a constant $C >0$ such that for all $k\geq0$ we have 
$\|\Exp[x^{k} -x^*]\|^2  \leq \beta^k C.$
\end{thm}

Using Theorem~\ref{theoremheavyball} and by a proper combination of the stepsize $\omega$ and the momentum parameter $\beta$, SHB enjoys an accelerated linear convergence rate in mean, \cite{loizou2017momentum}.
\begin{cor}
(i) If $ \omega= 1$ and $\beta= (1- \sqrt{0.99 \lambda_{\min}^+}) ^2$, then the iteration complexity of SHB becomes: $ \tilde{O}(\sqrt{1/ \lambda_{\min}^+})$.\\
(ii) If $ \omega= 1/\lambda_{\max}$ and $\beta= (1- \sqrt{ 0.99\lambda_{\min}^+/ \lambda_{\max}})^2$, then the iteration complexity of SHB becomes: $\tilde{O}(\sqrt{\lambda_{\max}/ \lambda_{\min}^+})$.
\end{cor}
\section{Randomized Gossip protocols with momentum}
\label{theProtocols}
Having presented SHB for solving the stochastic optimization problem \eqref{stoch_reform} and describing its sketch-and-project nature \eqref{SPmomentum}, let us now describe its behavior as a randomized gossip protocol when applied to solving the AC system $\bA x=0$, where $\bA\in |\cE| \times n$ is the incidence matrix of the network.  

Since $b=0$, method \eqref{SPmomentum} can be simplified to:
\begin{equation}
\label{momentumupdateb0}
x^{k+1}=\left[\bI -\omega \bA^\top \bH_k \bA\right] x^k + \beta(x^k - x^{k-1}).
\end{equation}

In the rest of this section we focus on  two special cases of \eqref{momentumupdateb0}: RK with momentum and RBK with momentum.

\subsection{Randomized Kaczmarz Gossip with momentum}
When RK is applied to solve an AC system $\bA x=0$, one recovers the famous pairwise gossip algorithm~\cite{boyd2006randomized}.  Algorithm~\ref{RKmomentum} describes how the relaxed variant of randomized Kaczmarz with momentum behaves as a gossip algorithm. See also Figure~\eqref{fig:mRK} for a graphical illustration of the method.

\begin{algorithm}[t!]
	\caption{mRK: Randomized Kaczmarz with momentum as a gossip algorithm}
	\label{RKmomentum}
	\small \small
	\begin{algorithmic}[1]
		\State {\bf Parameters:} Distribution $\mathcal{D}$ from which method samples matrices; stepsize/relaxation parameter $\omega \in \R$; heavy ball/momentum parameter $\beta$.
		\State {\bf Initialize:} $x^0 ,x^1 \in \R^n$
		\For{$k=1,2,\dots$} 
		\State Pick an edge $e=(i,j)$ following the distribution $\cD$
		\State The values of the nodes are updated as follows:
\begin{itemize}
\item Node $i$: $x_i^{k+1}= \frac{2-\omega}{2}x_i^k+ \frac{\omega}{2}x_j^k+\beta (x_i^k - x_i^{k-1})$
\item Node $j$: $x_j^{k+1}= \frac{2-\omega}{2} x_j^k+\frac{\omega}{2}x_i^k+\beta (x_j^k - x_j^{k-1})$
\item Any other node $\ell$: $x_\ell^{k+1}=x_\ell^k+\beta (x_\ell^k - x_\ell^{k-1})$
\end{itemize}
		\EndFor
		\State {\bf Output:} The last iterate $x^k$
	\end{algorithmic}
\end{algorithm}

\begin{rem}
In the special case with $\beta=0$ (zero momentum) only the two nodes of edge $e=(i,j)$ update their values. In this case the two nodes do not update their values to their exact average but to a convex combination that depends on the stepsize $\omega \in (0,2)$. To obtain the pairwise gossip algorithm of \cite{boyd2006randomized}, we should further choose $\omega=1$.
\end{rem}

\textbf{Distributed Nature of the Algorithm:} Here we highlight a few ways to implement mRK in a distributed fashion:
 \emph{Asynchronous pairwise broadcast gossip:}  In this protocol each node $i \in \cV$ of the network $\cG$ has a clock that ticks at the times of a rate 1 Poisson process. The inter-tick times are exponentially distributed, independent across nodes, and independent across time. This is equivalent to a global clock ticking at a rate $n$ Poisson process which wakes up an edge of the network at random. In particular, in this implementation mRK works as follows:  In the $k^{th}$ iteration (time slot) the clock of node $i$ ticks and node $i$ randomly contact one of its neighbors and simultaneously broadcast a signal to inform the nodes of the whole network that is updating (this signal does not contain any private information of node $i$). The two nodes $(i,j)$ share their information and update their private values following the update rule of Algorithm~\ref{RKmomentum} while all the other nodes updating their values using their own information. In each iteration only one pair of nodes exchange their private values. 
 
 \emph{Synchronous pairwise gossip:} In this protocol a single global clock is available to all nodes. The time is assumed to be slotted commonly across nodes and in each time slot only a pair of nodes of the network is randomly activated and exchange their information following the update rule of Algorithm~\ref{RKmomentum}.  The remaining not activated nodes update their values using their own last two private values.  Note that this implementation of mRK comes with the disadvantage that requires a central entity which choose the activate pair of nodes in each step.

 \emph{Asynchronous pairwise gossip with common counter:} 
The update rule of the nodes of the active pair $(i,j)$ in Algorithm~\ref{RKmomentum} can be rewritten as follows:
$$ x_i^{k+1}= x_i^k + \beta (x_i^k - x_i^{k-1}) + \tfrac{\omega}{2} (x_j^k -x_i^k)$$
$$x_j^{k+1}= x_j^k + \beta (x_j^k - x_j^{k-1}) + \tfrac{\omega}{2} (x_i^k -x_j^k)$$
In particular observe that in their update rule they have the expression $x_i^k + \beta (x_i^k - x_i^{k-1})$ which is precisely the update of all non activate nodes of the network. Thus if we assume that the nodes share a common counter that counts how many iterations take place and each node $i$ saves also the last iterate $k_i$ that it was activated then the algorithm can work in distributed fashion as follows:

Let us denote the number of total iterations (common counter) that becomes available to the activate nodes of each step as $K$ and let us define with $i_k=K-k_{i}$ the number of iterations between the current iterate and the last time that the $i^{th}$ node is picked (iteration $k_{i}$) then the update rule of the Algorithm~\ref{RKmomentum} can be equivalently expressed as: 
\begin{itemize}
\item  Pick an edge $e=(i,j)$ at random following $\cD$. 
\item The private values of the nodes  are updated as follows:
$$ x_i^{k+1}= i_k \left[x_i^k + \beta (x_i^k - x_i^{k-1}) \right]+ \tfrac{\omega}{2} (x_j^k -x_i^k)$$
$$x_j^{k+1}= j_k \left[x_j^k + \beta (x_j^k - x_j^{k-1}) \right] + \tfrac{\omega}{2} (x_i^k -x_j^k)$$
$k_i = k_j =k+1$; for any other node $\ell$: $x_\ell^{k+1}=x_\ell^k$
\end{itemize}

\begin{figure}[t!]
\begin{minipage}[b]{1.0\linewidth}
  \centering
  \centerline{\includegraphics[scale=0.4]{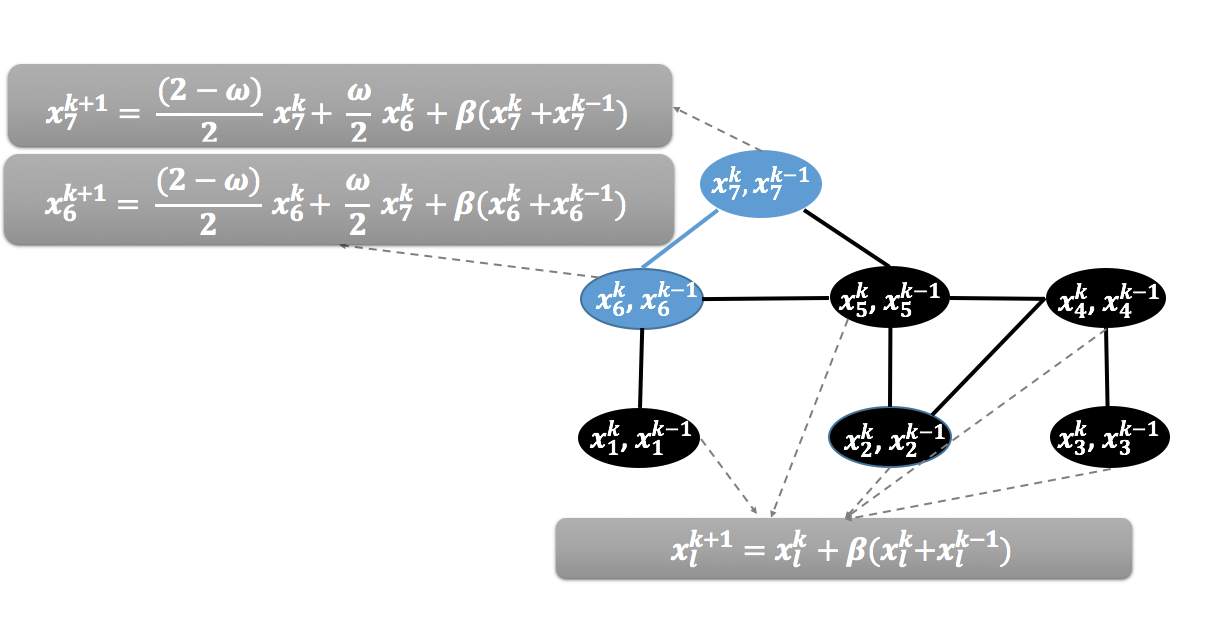}}
  \caption{\footnotesize Example of how mRK works as gossip algorithm. In the presented network the edge that connects nodes $6$ and $7$ is randomly selected. The pair of nodes exchange their information and update their values following the update rule of the Algorithm~\ref{RKmomentum} while the rest of the nodes, $\ell \in [5]$, update their values using only their own previous private values.}
  \label{fig:mRK}
\end{minipage}
\end{figure}

\subsection{Connection with the accelerated gossip algorithm}
\label{connectionOfAcceleratedMethods}
In the randomized gossip literature there in one particular method closely related to our approach. It was first proposed in \cite{cao2006accelerated} and its analysis under strong conditions was presented in \cite{liu2013analysis}.
In this paper local memory is exploited by installing shift registers at each agent. In particular we are interested in the case of just two registers where the first stores the agent's current value and the second the agent's value before the latest update. The algorithm can be described as follows. Suppose that edge $e=(i,j)$ is chosen at time $k$. Then,
\begin{itemize}
\item Node $i$: $x_i^{k+1}= \omega(\frac{x_i^k+x_j^k}{2})+(1-\omega)x_i^{k-1}$
\item Node $j$: $x_i^{k+1}=  \omega(\frac{x_i^k+x_j^k}{2})+(1-\omega)x_j^{k-1}$
\item Any other node $\ell$: $x_\ell^{k+1}=x_\ell^k$
\end{itemize}
where $\omega \in [1,2)$. The method was analyzed in \cite{liu2013analysis} under a strong assumption on the probabilities of choosing the pair of nodes that as the authors mentioned is unrealistic in practical scenarios, and for networks like the random geometric graphs. At this point we should highlight that the results presented in \cite{loizou2017momentum} hold for essentially any distribution $\cD$ and as a result such a problem cannot occur.

Note also that  if we choose $\beta=\omega-1$ in the update rule of Algorithm~\ref{RKmomentum}, then our method is simplified to:
\begin{itemize}
\item Node $i$: $x_i^{k+1}= \omega(\frac{x_i^k+x_j^k}{2})+(1-\omega)x_i^{k-1}$
\item Node $j$: $x_i^{k+1}=  \omega(\frac{x_i^k+x_j^k}{2})+(1-\omega)x_j^{k-1}$
\item Any other node $\ell$: $x_\ell^{k+1}=\omega x_\ell^k+(1-\omega)x_\ell^{k-1}$
\end{itemize}

In order to apply Theorem~\ref{L2}, we need to assume that $0< \omega < 2$ and $\beta=\omega-1 \geq 0$ which also means that $\omega \in [1,2)$. Thus for $\omega \in [1,2)$ and momentum parameter $\beta=\omega-1$ it is easy to see that our approach is very similar to the shift-register algorithm. Both methods update the selected pair of nodes in the same way. However, in our case the other nodes of the network  do not remain idle but instead also update their values using their own previous information.

Using the  momentum matrix $\bB=Diag(b_{11},b_{22},\dots, b_{nn})$, the two algorithms above can be expressed as:
\begin{equation}
\label{updatewithB}
x^{k+1} = x^k - \tfrac{\omega}{2} (x_i^k-x_j^k)(e_i-e_j) + \bB (x^k-x^{k-1}).
\end{equation}
In particular, in our algorithm every element on the diagonal is equal to $\beta=\omega-1$, while in \cite{cao2006accelerated} all values on the diagonal are zeros except for the two values $b_{ii}=b_{jj}=\omega-1$.

\begin{rem}
The shift register case and our algorithm can be seen as  two limit cases of the update rule \eqref{updatewithB}. In particular, the shift register method uses only two non-zero diagonal elements in $\bB$, while our method has a full diagonal. We believe that further methods can be developed in the future by exploring the cases where more than two but not all elements of the diagonal matrix $\bB$ are non-zero. It might be possible to obtain better convergence  if one carefully chooses these values based on the network topology. We leave this as an open problem for future research.
\end{rem}

\subsection{Randomized block Kaczmarz gossip with momentum}

Recall that Theorem~\ref{TheoremRBK} says how RBK (with no momentum and no relaxation) can be interpreted as a gossip algorithm. Now we  use this result to explain how relaxed RBK with momentum works. Note that the update rule of RBK with momentum can be rewritten as follows:
\begin{equation}
\label{updateRBK2}
x^{k+1} \overset{\eqref{momentumupdateb0}}{=} \omega (\bI- \bA^ \top \bH_k \bA) x^k+(1-\omega)x^k +\beta(x^k-x^{k-1}),
\end{equation}
where $(\bI - \bA^ \top \bH_k \bA) x^k$ is the update rule of  RBK   \eqref{updateSkProj}.

Thus, in analogy  to the simple RBK, in the $k^{th}$ step, a random set of edges is selected and $q \leq n$ connected components are formed as a result. This includes the connected components that belong to both sub-graph $\cG_k$ and also the singleton connected components (nodes outside the $\cG_k$). Let us define the set of the nodes that belong in the $r \in [q]$ connected component at the $k^{th}$ step $\cV_r^k$, such that $\cV= \cup_{r\in [q]} \cV_r^k$ and $|\cV|=\sum_{r=1}^{q} |\cV_r^k|$ for any $k>0$. 

Using the update rule \eqref{updateRBK2}, Algorithm~\ref{RBKmomentum} shows how mRBK is updating the private values of the nodes of the network (see also Figure~\ref{fig:RBK} for the graphical interpretation).

\begin{algorithm}[t!]
	\caption{Randomized Block Kaczmarz Gossip with momentum}
	\label{RBKmomentum}
	\small \small
	\begin{algorithmic}[1]
		\State {\bf Parameters:} Distribution $\mathcal{D}$ from which method samples matrices;  stepsize/relaxation parameter $\omega \in \R$;  heavy ball/momentum parameter $\beta$.
		\State {\bf Initialize:} $x^0,x^1 \in \R^n$
		\For{$k=1,2,...$} 
		\State Select a random set of edges $\cS \subseteq \cE$
		\State Form subgraph $\cG_k$ of $\cG$ from the selected  edges 
		\State Node values  are updated as follows:
\begin{itemize}
\item For each connected component $\cV_r^k$ of $\cG_k$, replace the values of its nodes with: 
\begin{equation}
\label{updateruelblock}
x_i^{k+1}=\omega \tfrac{\sum_{j \in \cV_r^k} x_j^{k}}{|\cV_r^k|} +(1-\omega)x_i^k+\beta (x_i^k-x_i^{k-1})
\end{equation}
\item Any other node $\ell$: $x_\ell^{k+1}=x_\ell^k+\beta (x_\ell^k - x_\ell^{k-1})$
\end{itemize}
		\EndFor
		\State {\bf Output:} The last iterate $x^k$
	\end{algorithmic}
\end{algorithm}

Note that in the update rule of mRBK the nodes that are not attached to a selected edge (do not belong in the sub-graph $\cG_k$) update their values via $x_\ell^{k+1}=x_\ell^k+\beta (x_\ell^k - x_\ell^{k-1})$. By considering these nodes as singleton connected components their update rule is exactly the same with the nodes of sub-graph $\cG_k$. This is easy to see as follows:
\begin{eqnarray}
x_\ell^{k+1}&=&\omega \tfrac{\sum_{j \in \cV_r^k} x_j^{k}}{|\cV_r^k|} +(1-\omega)x_\ell^k+\beta (x_\ell^k-x_\ell^{k-1})\notag\\
&=&\omega x_\ell^k +(1-\omega)x_\ell^k+\beta (x_\ell^k-x_\ell^{k-1})\notag\\
&=& x_\ell^k+\beta (x_\ell^k - x_\ell^{k-1}).
\end{eqnarray}

\begin{rem}
In the special case that only one edge is selected in each iteration ($\bS_k \in \R^{m \times 1}$)  the update rule of mRBK is simplified to the update rule of mRK. In this case the sub-graph $\cG_k$ is the pair of the two selected edges.  
\end{rem}

\begin{rem}
In \cite{LoizouRichtarik} it was shown that several existing gossip protocols for solving the average consensus problem are special cases of the simple RBK (Theorem~\ref{TheoremRBK}). For example two gossip algorithms that can be cast as special cases of the simple RBK are the path averaging proposed in \cite{benezit2010order} and the clique gossiping \cite{liu2017clique}. In  path averaging, in each iteration a path of nodes is selected and its nodes update their values to their exact average ($\omega=1$). In clique gossiping, the network is already divided into cliques and a through a random procedure a clique is activated and the nodes of it update their values to their exact average ($\omega=1$). Since mRBK contains simple RBK as a special case for $\beta=0$, we expect that these special protocols can  also be accelerated with the addition of momentum parameter $\beta \in (0,1)$.
\end{rem}

\begin{figure}[t!]
\begin{minipage}[b]{1.0\linewidth}
  \centering
  \centerline{\includegraphics[scale=0.41]{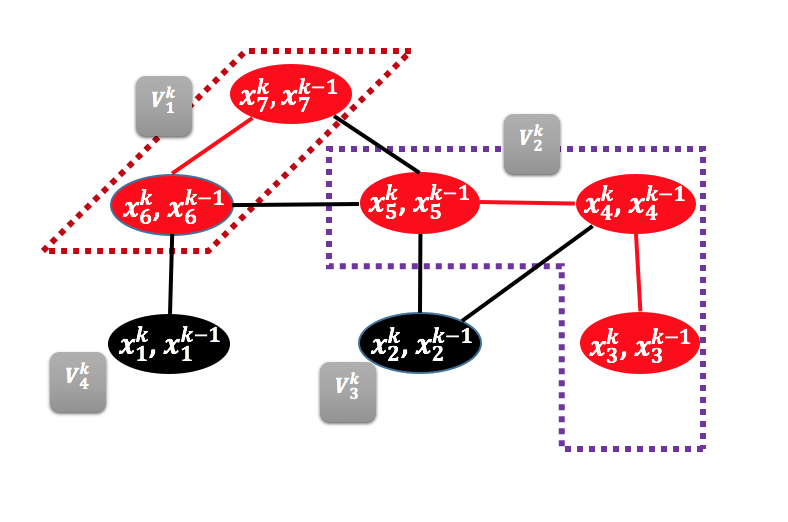}}
  \caption{\footnotesize Example of how the mRBK method works as gossip algorithm. In the presented network in the $k^{th}$ iteration the red edges are randomly chosen and they form subgraph $\cG_k$(from the red edges) and also four connected component. In this figure $V_1^k$ and $V_2^k$ are the two connected components that belong in the subgraph $\cG_k$ while $V_3^k$ and $V_4^k$ are the singleton connected components. Then the nodes update their values by communicate with the other nodes of their connected component using the update rule \eqref{updateruelblock}. For example the node number 5 that belongs in the connected component $V_2^k$ will update its value using the values of node 4 and 3 that also belong in the same component as follows:
$x_5^{k+1}=\omega \tfrac{x_3^k+x_4^k+x_5^k}{3} +(1-\omega)x_5^k+\beta (x_5^k-x_5^{k-1})$.}
  \label{fig:RBK}
\end{minipage}
\end{figure}

\subsection{Mass preservation}
One of the key properties of some of the most efficient randomized gossip algorithms is mass preservation. If a gossip algorithm has this property it means that the sum (and as a result the average) of the private values of the nodes remains fixed during the iterative procedure. That is, $\textstyle \sum_{i=1}^{n}x_i^{k}=\sum_{i=1}^{n}x_i^{0}, \quad \forall k \geq 1.$ The original pairwise gossip algorithm proposed in \cite{boyd2006randomized} satisfied the mass preservation property, while exisiting accelerated gossip algorithms \cite{cao2006accelerated,liu2013analysis}  preserving a scaled sum.  

In this section we show that the two proposed protocols presented above also have a mass preservation property. In particular, we prove mass preservation for the case of the block randomized gossip protocol (Algorithm~\ref{RBKmomentum}) with momentum. This is sufficient since the Kaczmarz gossip with momentum (mRK) can be cast as special case.

\begin{thm}
Assume that $x^0=x^1$. That is, the two registers of each node have the same initial value.  Then for the Algorithms~\ref{RKmomentum} and \ref{RBKmomentum} we have $\sum_{i=1}^{n}x_i^k=\sum_{i=1}^{n}c_i$ for any $k\geq 0$ and as a result, $\tfrac{1}{n}\sum_{i=1}^{n}x_i^k=\bar{c}$. 
\end{thm}

\begin{proof}
We prove the result for the more general Algorithm~\ref{RBKmomentum}. Assume that in the $k^{th}$ step of the method $q$ connected components are formed.  Let the set of the nodes of each connected component be $\cV_r^k$ so that $\cV= \cup_{r=\{1,2,...q\}} \cV_r^k$ and $|\cV|=\sum_{{r}=1}^{q} |\cV_r^k|$ for any $k>0$.  Thus:
\begin{equation}
\label{generalsum}
\textstyle \sum_{i=1}^{n}x_i^{k+1}=\sum_{i \in \cV_1^k} x_i^{k+1} +\dots + \sum_{i \in \cV_q^k} x_i^{k+1}
\end{equation}
Let us first focus, without loss of generality, on  connected component $r \in [q]$ and simplify the expression for the sum of its nodes:
$ \sum_{i\in \cV_r^k} x_i^{k+1}
\overset{\eqref{updateruelblock}}= \textstyle \sum_{i \in \cV_r^k} \omega \tfrac{\sum_{j \in \cV_r^k} x_j^{k}}{|\cV_r^k|} +   (1-\omega) \sum_{i \in \cV_r^k} x_i^k  +\beta \sum_{i \in \cV_r^k}  (x_i^k-x_i^{k-1})
 =|\cV_r^k| \tfrac{\omega \sum_{j \in \cV_r^k} x_j^{k}}{|\cV_r^k|}+ (1-\omega) \sum_{i \in \cV_r^k} x_i^k 
 +\beta \sum_{i \in \cV_r^k}  (x_i^k-x_i^{k-1})
  =(1+\beta) \sum_{i \in \cV_r^k}x_i^k-\beta \sum_{i \in \cV_r^k}x_i^{k-1}
$. By substituting this for all $r \in [q]$ into the right hand side of \eqref{generalsum} and from the fact that $\cV= \cup_{r\in [q]}
 \cV_r^k$, we get
$ \sum_{i=1}^{n}x_i^{k+1}
= (1+\beta) \sum_{i=1}^{n}x_i^k-\beta \sum_{i=1}^{n} x_i^{k-1}.
$
Since $x^0=x^1$, we have $\sum_{i=1}^{n}x_i^{0}=\sum_{i=1}^{n}x_i^{1}$, and as a result $
\sum_{i=1}^{n}x_i^{k}
= \sum_{i=1}^{n}x_i^{0}$ for all $ k \geq 0$.
\end{proof}

\section{Numerical Evaluation}
\label{experiments}
We devote this section to experimentally evaluate the performance of the proposed gossip algorithms: mRK and mRBK. In particular we perform three experiments. In the first two we focus on the performance of the mRK, while in the last one on its block variant mRBK.
In comparing the methods with their momentum variants we use the relative error measure $\|x^k-x^*\|^2 / \|x^0-x^*\|^2 $ where the starting vectors of values $x^0=x^1=c$ are taken to be always Gaussian vectors.  For all of our experiments the horizontal axis represents the number of iterations. The networks used in the experiments are the cycle (ring graph), the 2-dimension grid and the randomized geometric graph (RGG) with radius $r=\sqrt{\log(n)/n}$. Code was written in Julia 0.6.3. 



\subsection{Impact of momentum parameter on mRK}
Recall that in the simple pairwise gossip algorithm the two nodes that exchange information update their values to their exact average while all the other nodes remain idle. In our framework this method can be cast as special case of mRK when $\beta=0$ and $\omega=1$. In this experiment we keep always the stepsize to be $\omega=1$ which means that the pair of the chosen nodes update their values to their exact average. We show that by choosing a suitable momentum parameter $\beta \in (0,1)$ we can have faster convergence for all networks under study. See Figure~\ref{mRKomega1} for more details.

\begin{figure}[t!]
\centering
\begin{subfigure}{.25\textwidth}
  \centering
  \includegraphics[width=1\linewidth]{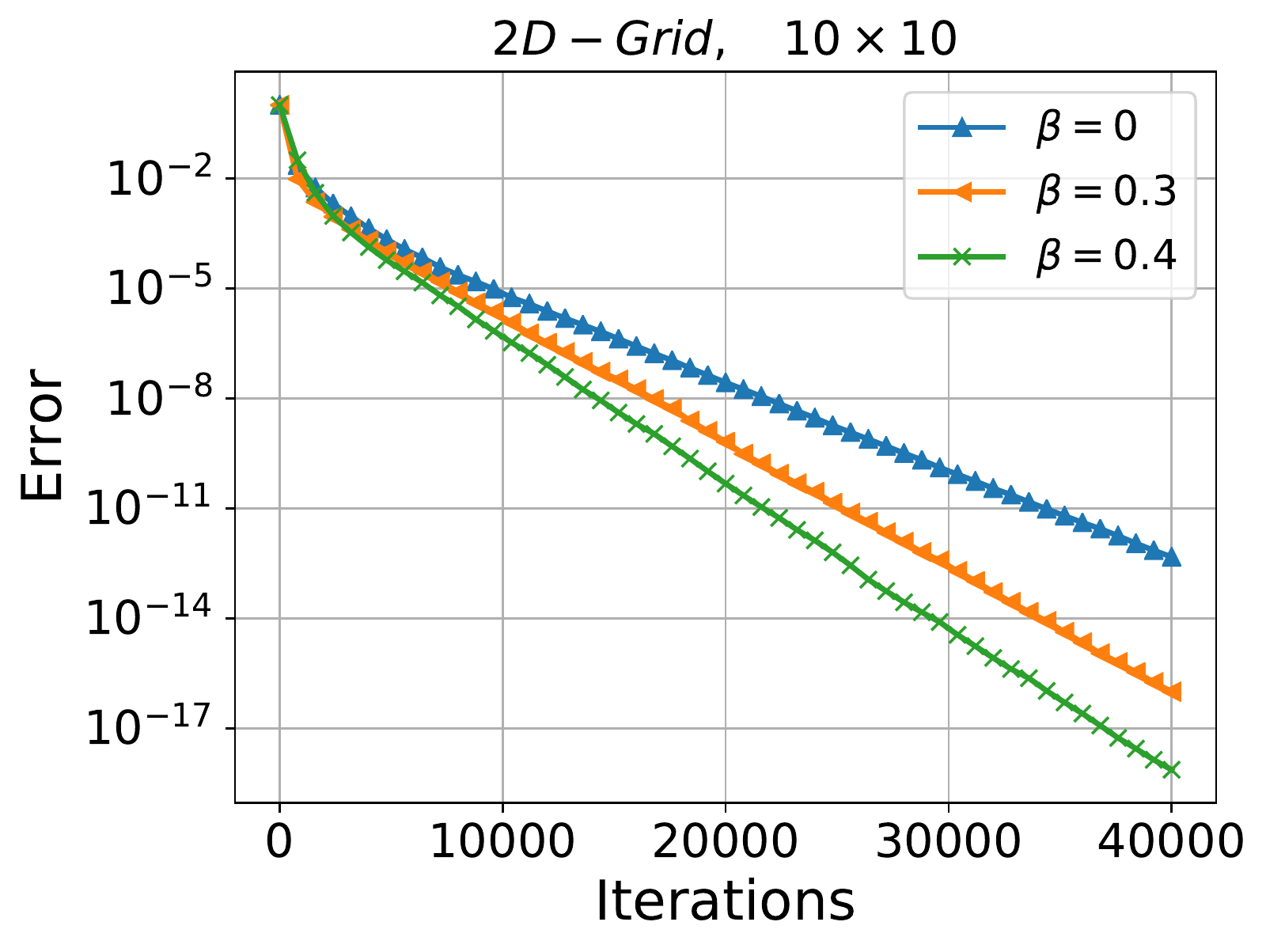}
  \label{fig:sub1}
\end{subfigure}%
\begin{subfigure}{.25\textwidth}
  \centering
  \includegraphics[width=1\linewidth]{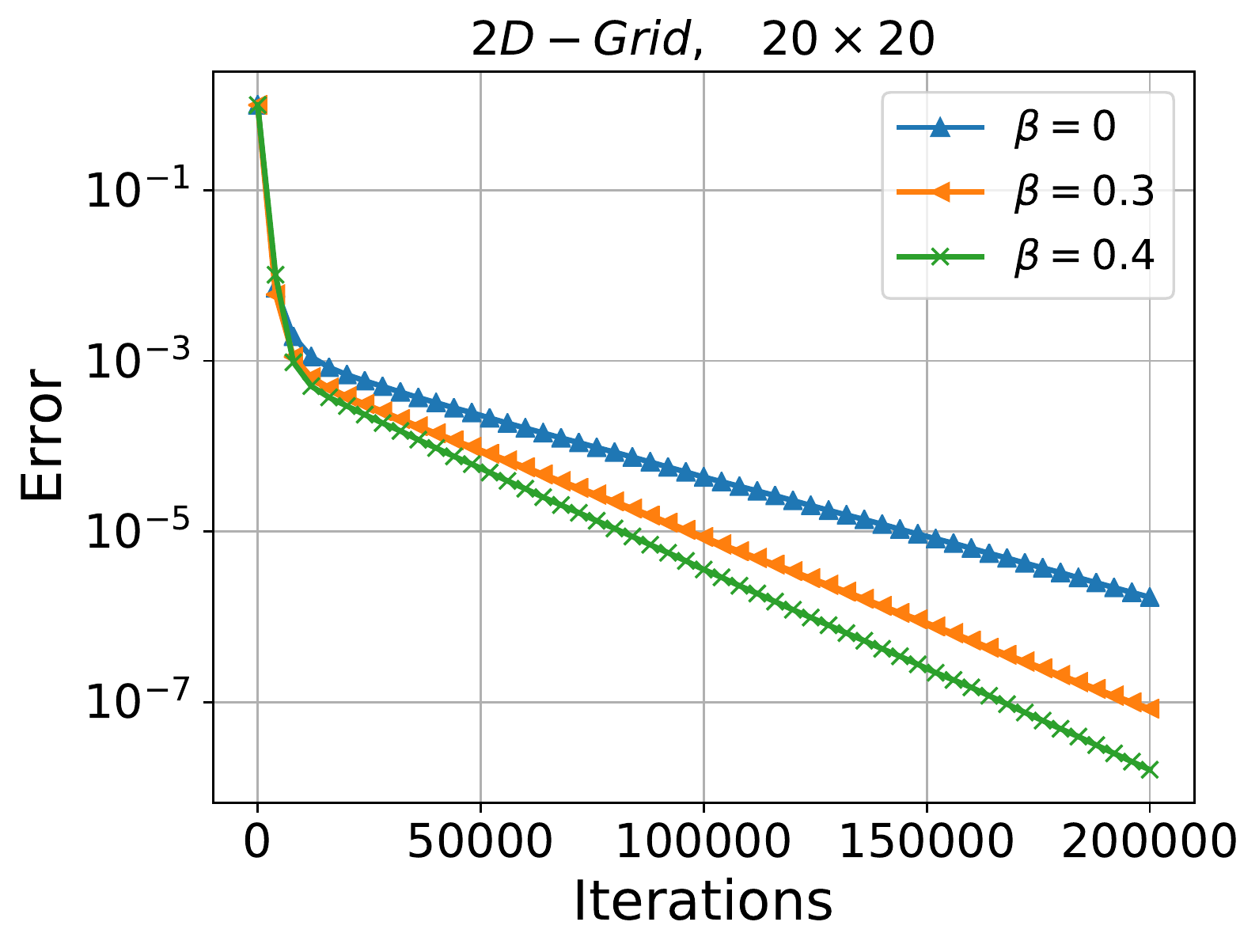}
  \label{fig:sub2}
\end{subfigure}\\
\begin{subfigure}{.25\textwidth}
  \centering
  \includegraphics[width=1\linewidth]{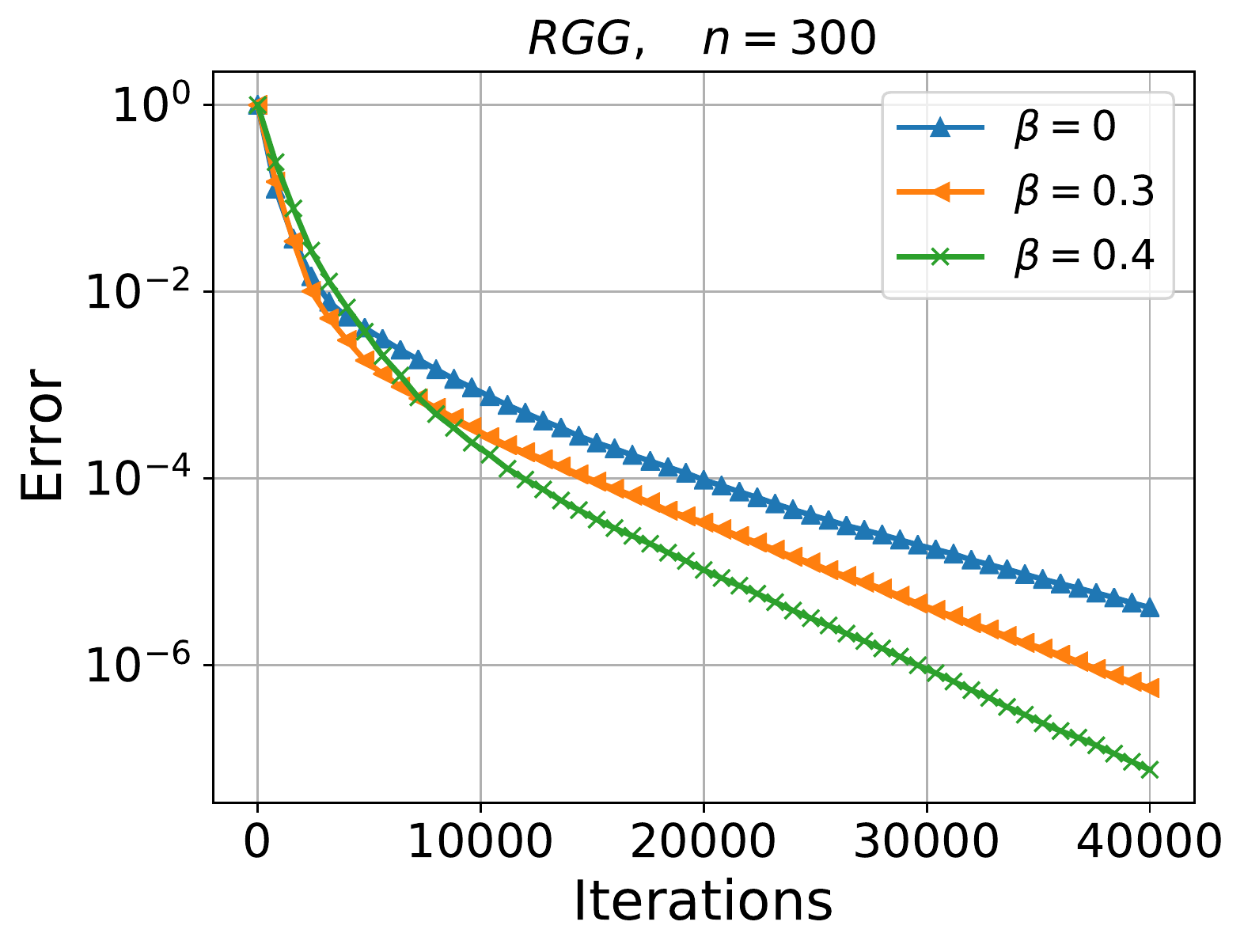}
  \label{fig:sub1}
\end{subfigure}%
\begin{subfigure}{.25\textwidth}
  \centering
  \includegraphics[width=1\linewidth]{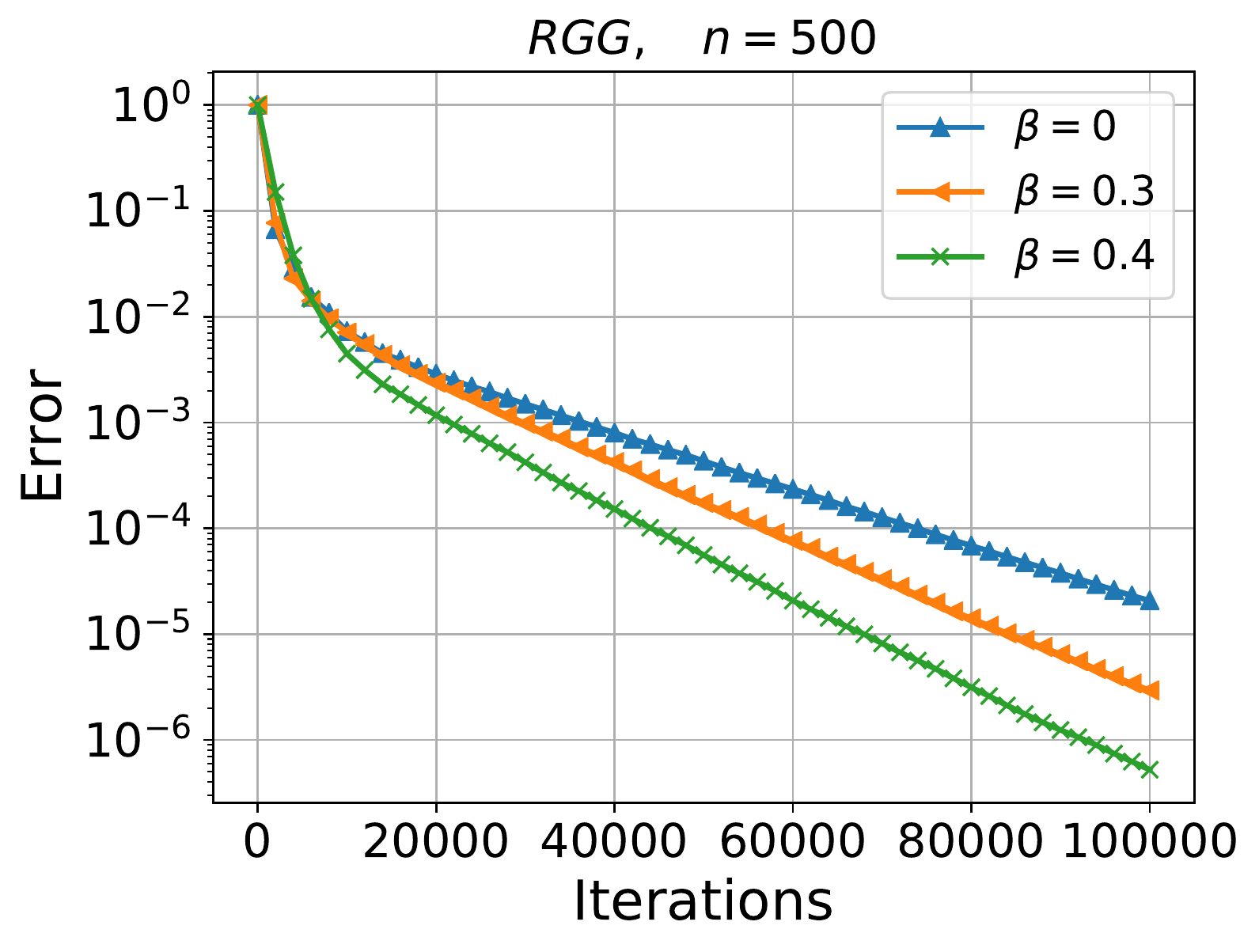}
  \label{fig:sub2}
\end{subfigure}\\
\begin{subfigure}{.25\textwidth}
  \centering
  \includegraphics[width=1\linewidth]{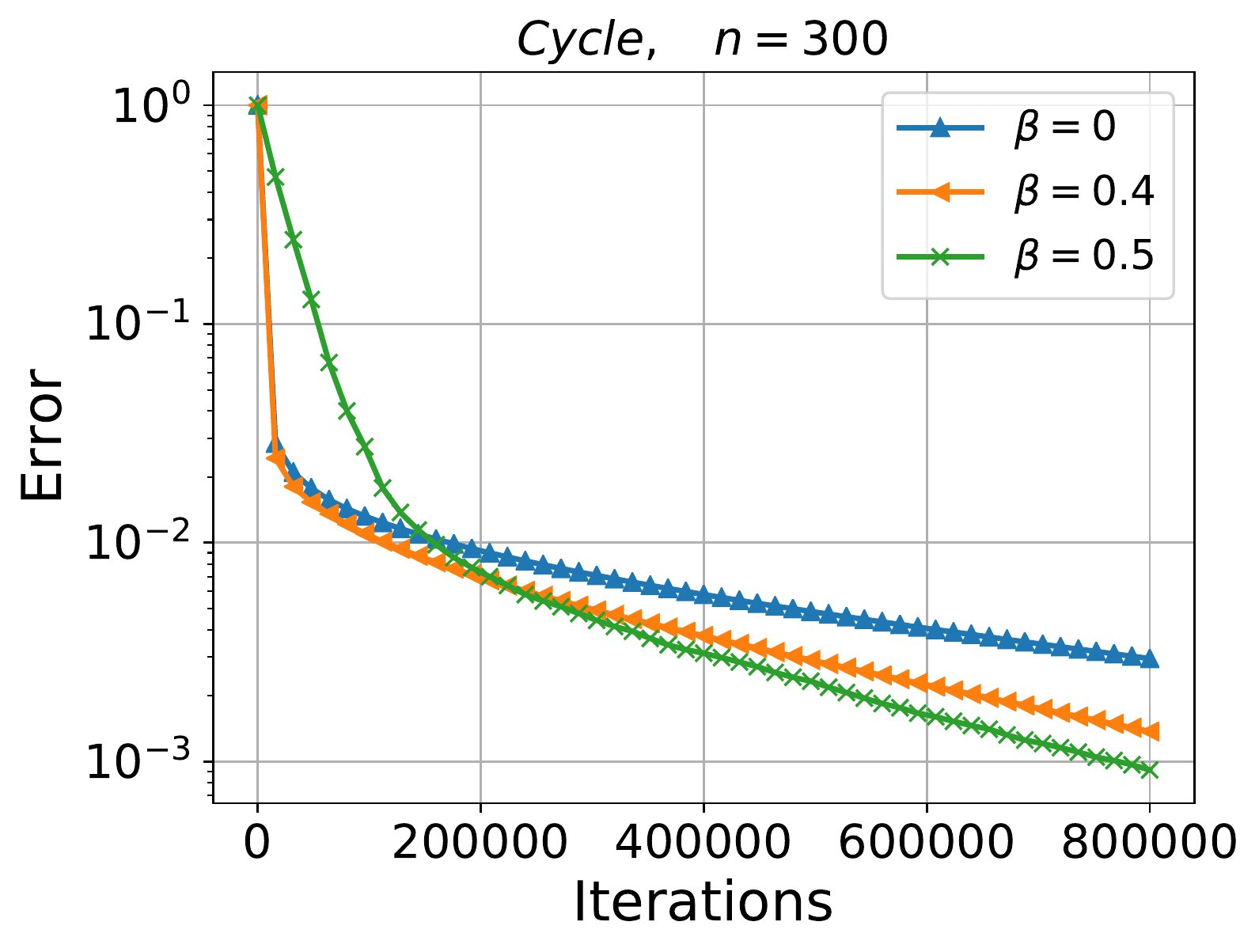}
\end{subfigure}%
\begin{subfigure}{.25\textwidth}
  \centering
  \includegraphics[width=1\linewidth]{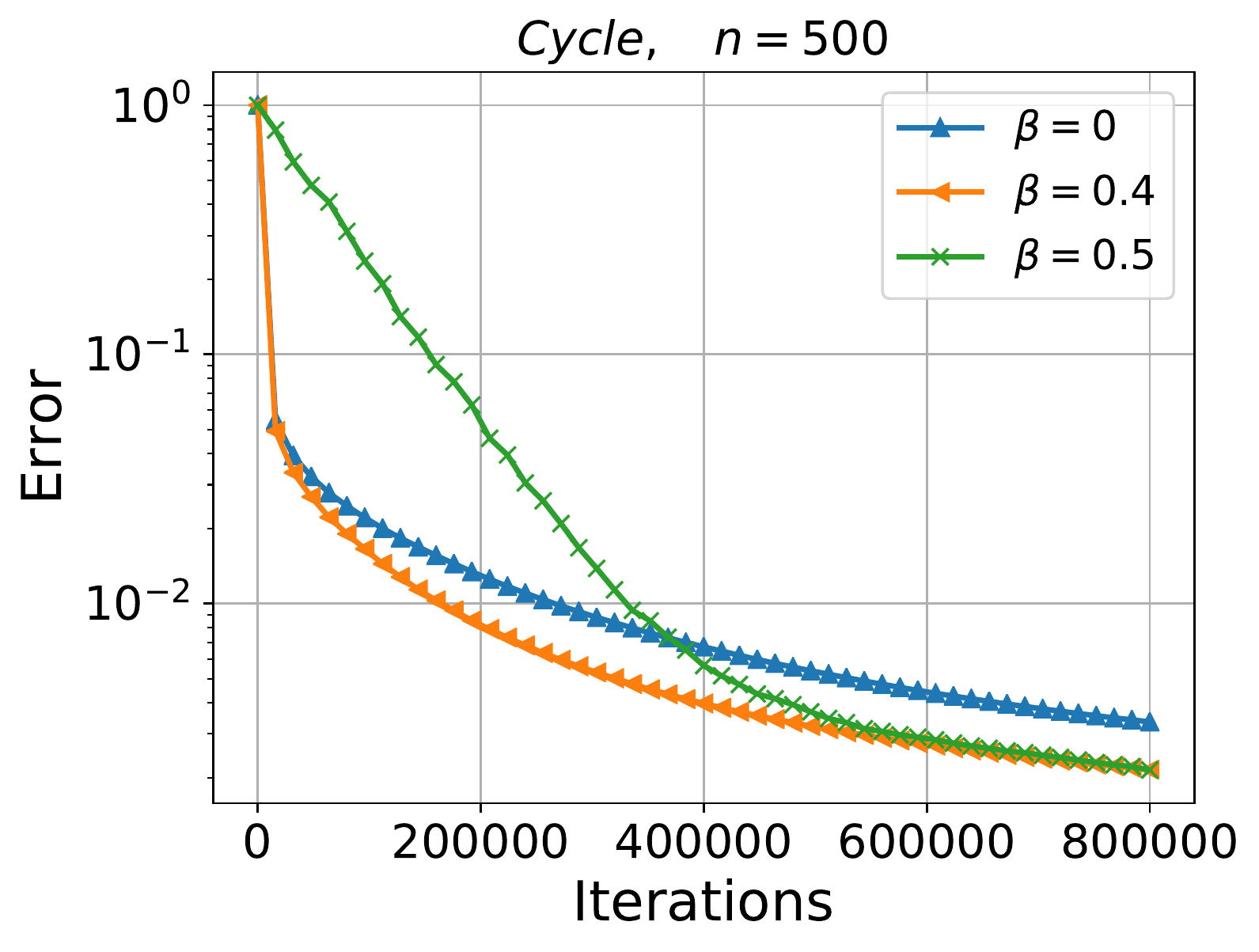}
\end{subfigure}
\caption{\footnotesize Performance of mRK for fixed step-size $\omega=1$ and several momentum parameters $\beta$ in a cycle, 2-dimension grid and RGG. The choice $\beta=0$ corresponds to the randomized pairwise gossip algorithm proposed in \cite{boyd2006randomized}; The $n$ in the title of each plot indicates the number of nodes of the network. For the grid graph this is $n \times n$.}
\label{mRKomega1}
\end{figure}

\subsection{Comparison with the Shift-Register}
In this experiment we compare mRK with the shift register case when we choose the $\omega$ and $\beta$ in such a way in order to satisfy the connection establish in Section~\ref{connectionOfAcceleratedMethods}. That is, we choose $\beta=\omega-1$ for any choice of $\omega \in (1,2)$. Observe that in all plots of Figure~\ref{shiftregister} our algorithm outperform the corresponding shift-register case. 
\begin{figure}[t!]
\centering
\begin{subfigure}{.23\textwidth}
  \centering
  \includegraphics[width=1\linewidth]{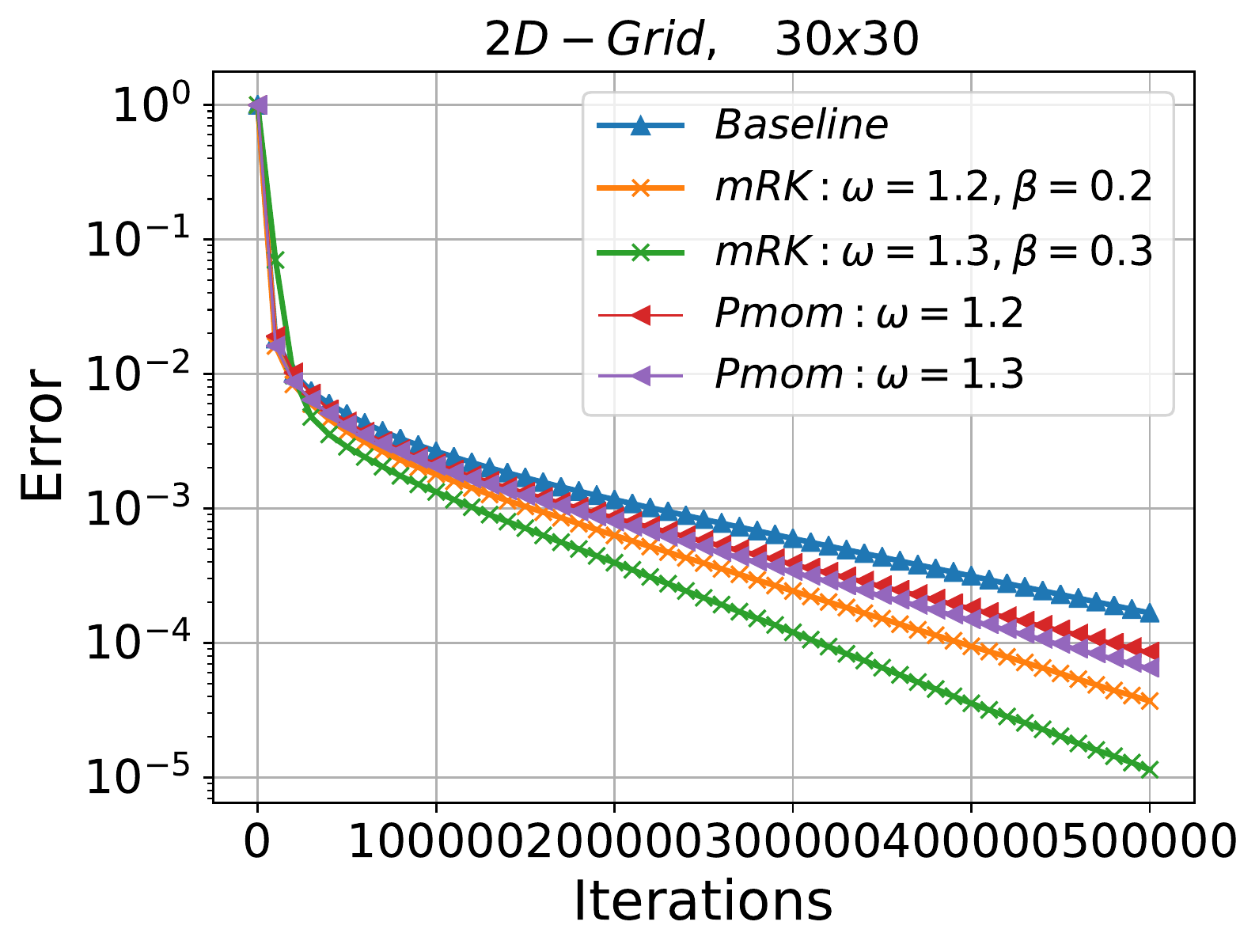}
\end{subfigure}
\begin{subfigure}{.23\textwidth}
  \centering
  \includegraphics[width=1\linewidth]{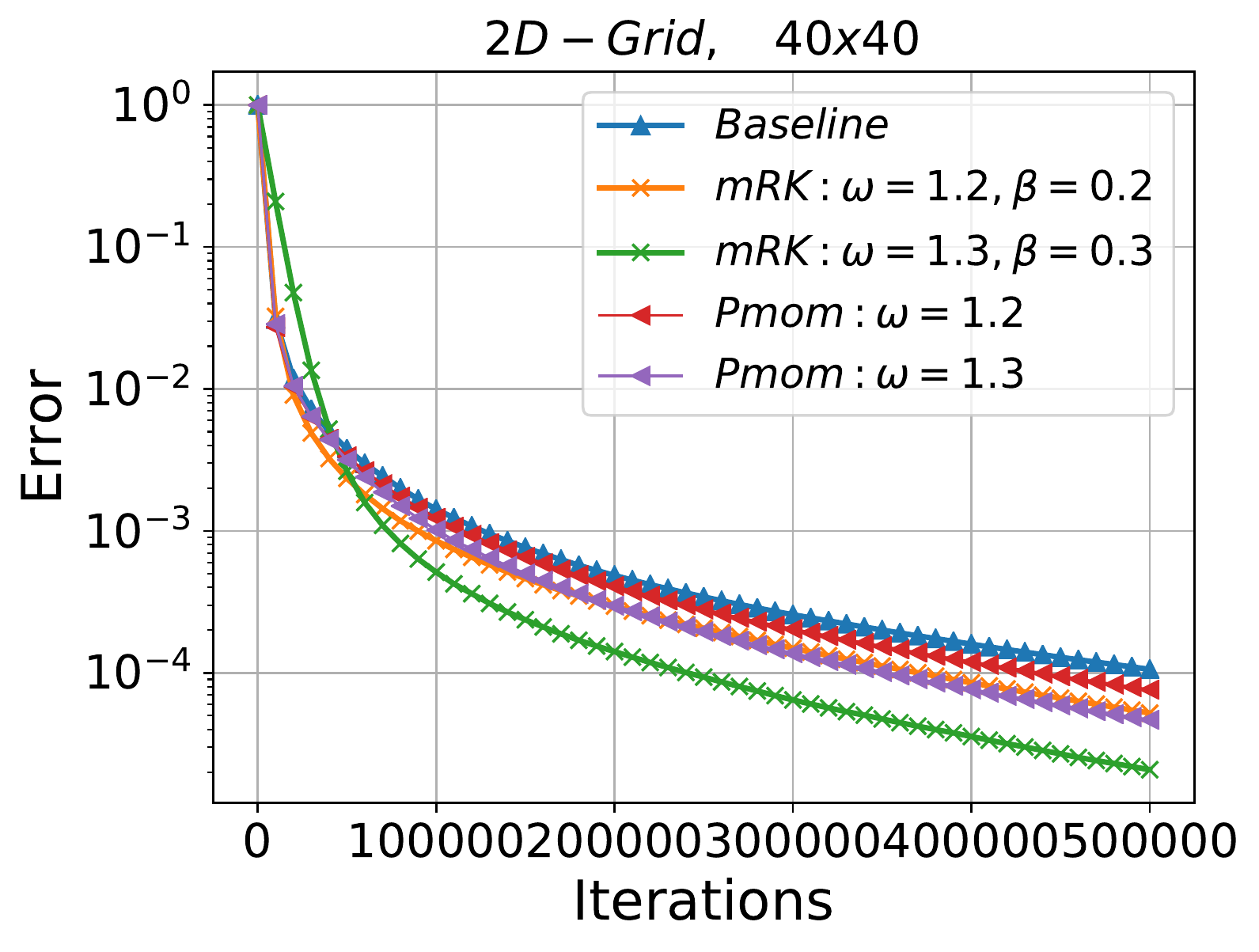}
\end{subfigure}\\
\begin{subfigure}{.23\textwidth}
  \centering
  \includegraphics[width=1\linewidth]{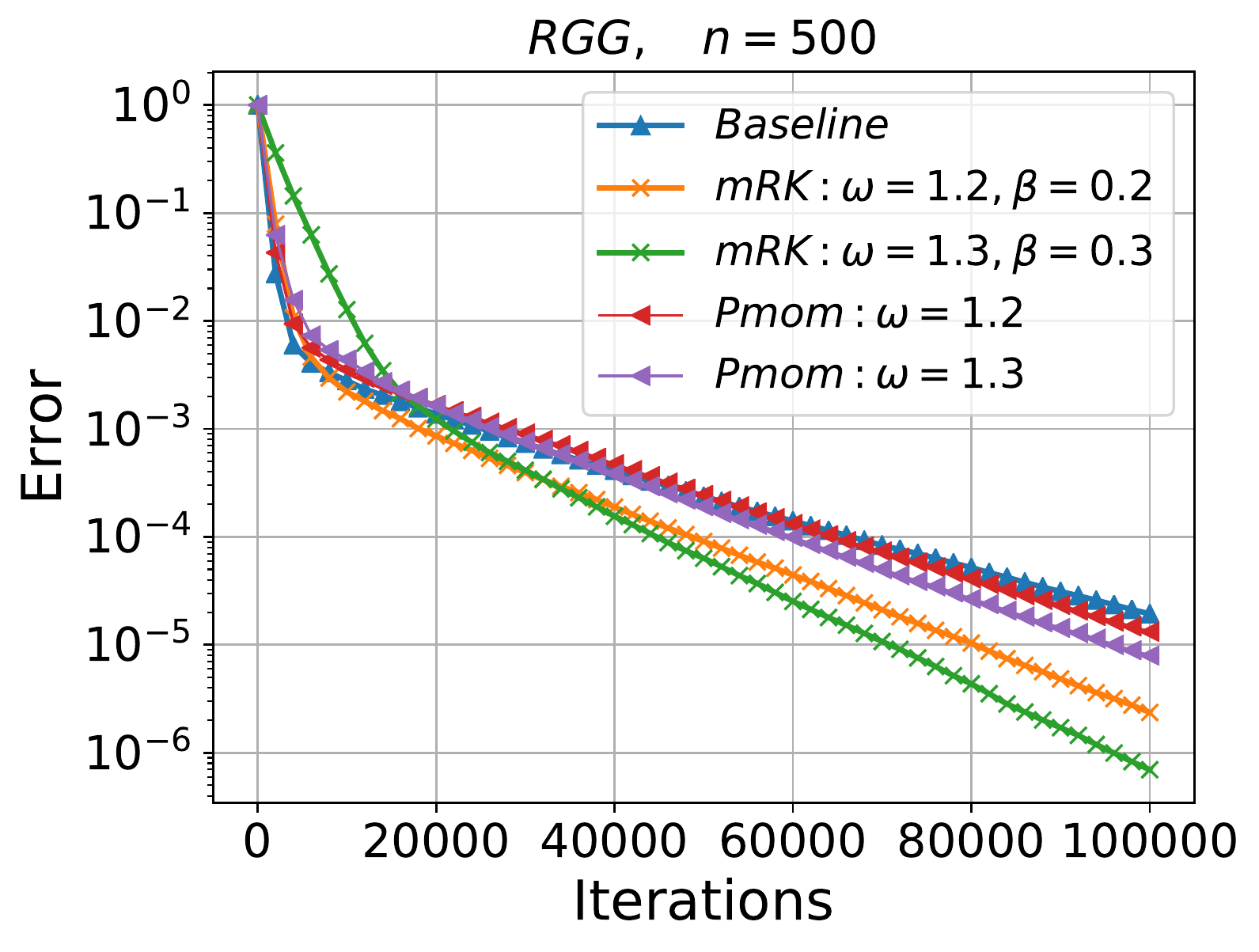}
\end{subfigure}%
\begin{subfigure}{.23\textwidth}
  \centering
  \includegraphics[width=1\linewidth]{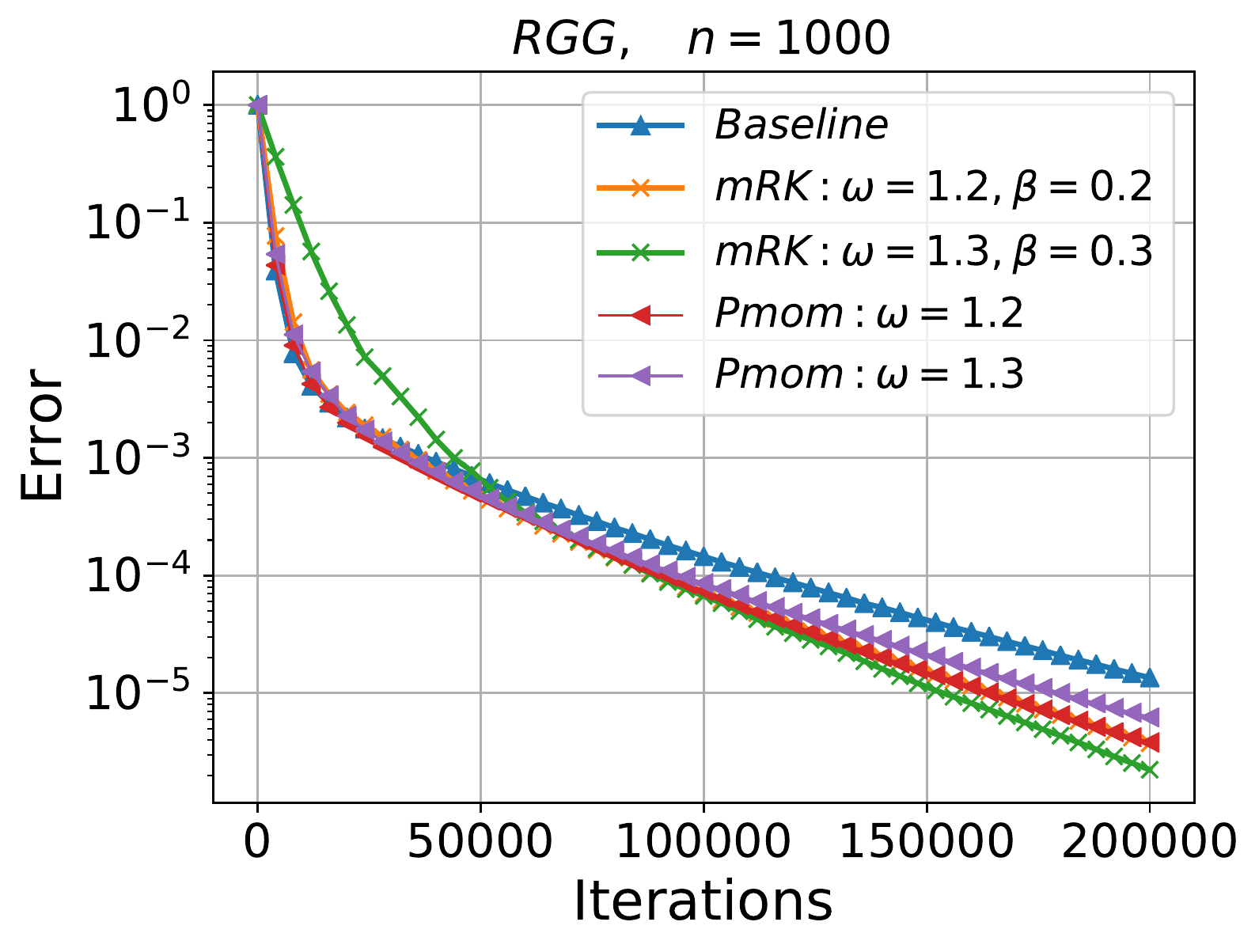}
\end{subfigure}\\
\begin{subfigure}{.23\textwidth}
  \centering
  \includegraphics[width=1\linewidth]{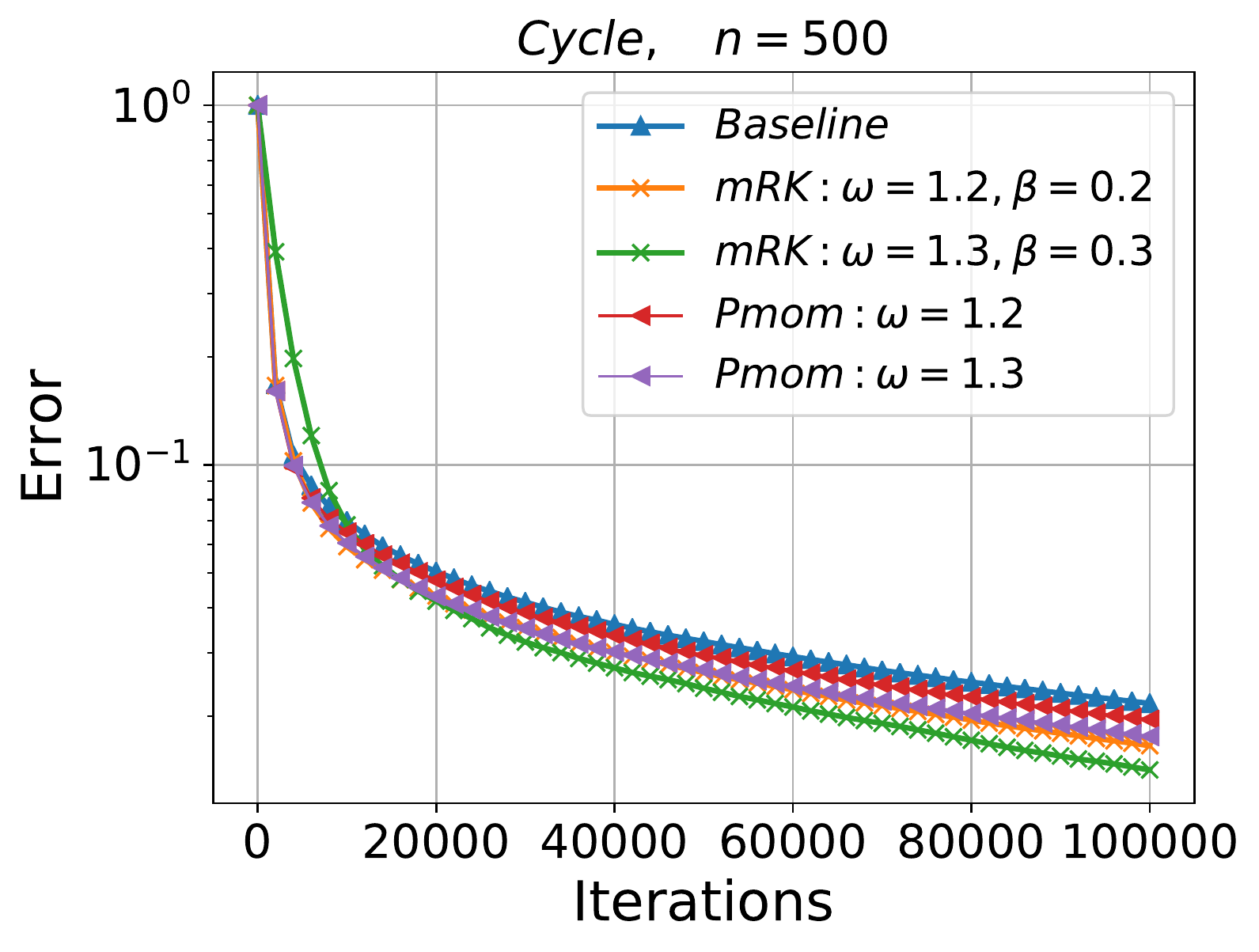}
\end{subfigure}
\begin{subfigure}{.23\textwidth}
  \centering
  \includegraphics[width=1\linewidth]{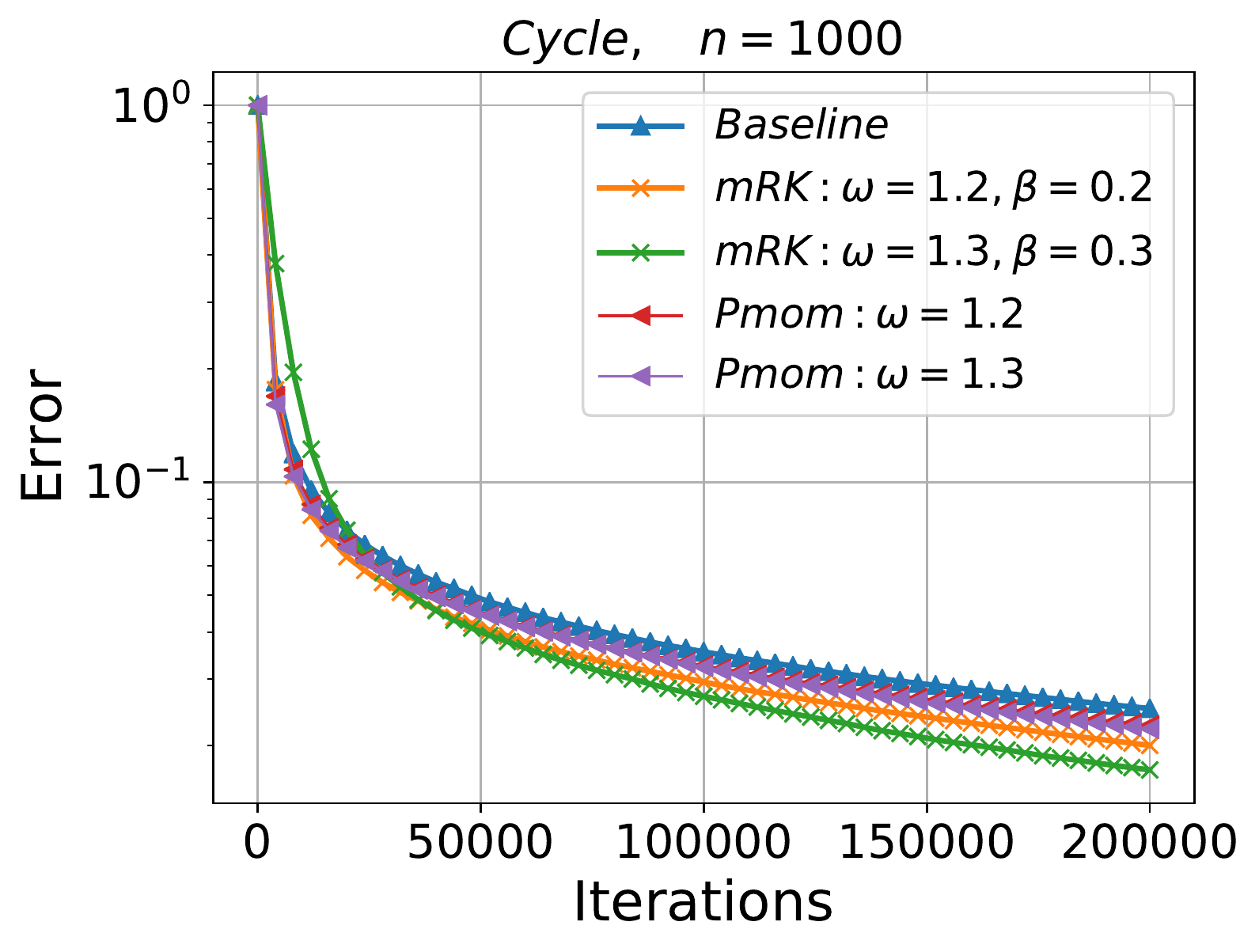}
\end{subfigure}%
\caption{\footnotesize Comparison of mRK with the pairwise momentum method (Pmom),  shift-register algorithm proposed in \cite{liu2013analysis}. For fair comparison we take always $\beta=\omega-1$ for our algorithm and the stepsizes are chosen to be either $\omega= 1.2$ or $\omega=1.3$.  The baseline method is the simple not accelerated randomized pairwise gossip algorithm from \cite{boyd2006randomized}. The $n$ in the title of each plot indicates the number of nodes of the network. For the grid graph this is $n \times n$.}
\label{shiftregister}
\end{figure}

\subsection{Impact of momentum parameter on mRBK}
In this experiment our goal is to show that the addition of momentum accelerates the RBK gossip algorithm proposed in \cite{LoizouRichtarik}. Without loss of generality we choose the block size to be always equal to $\tau=5$. That is the random matrix $\bS_k\sim \cD$ in the update rule of mRBK is always a $m \times 5$ column submatrix of the indetity $m \times m$ matrix. Thus, in each iteration $5$ edges of the network are chosen to form the subgraph $\cG_k$ and the values of the nodes are updated according to Algorithm~\ref{RBKmomentum}. Note that similar plots can be obtained for any choice of block size. We run all algorithms with fixed stepsize $\omega=1$. It is obvious that by choosing a suitable momentum parameter $\beta \in (0,1)$ we have faster convergence than when $\beta =0$, for all networks under study. See Figure~\ref{RBKfigures} for more details.

\begin{figure}[t!]
\centering
\begin{subfigure}{.25\textwidth}
  \centering
  \includegraphics[width=1\linewidth]{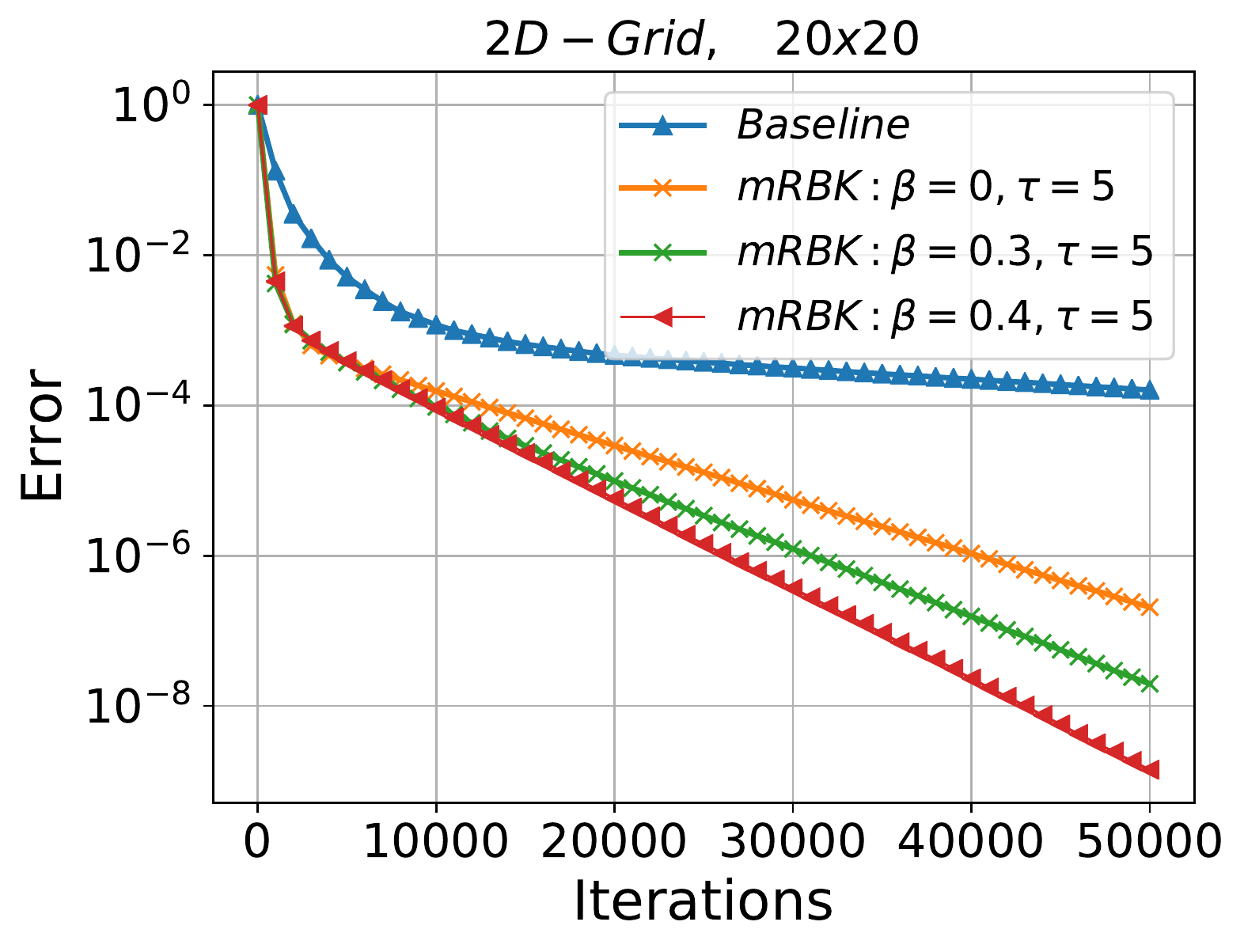}
\end{subfigure}%
\begin{subfigure}{.25\textwidth}
  \centering
  \includegraphics[width=1\linewidth]{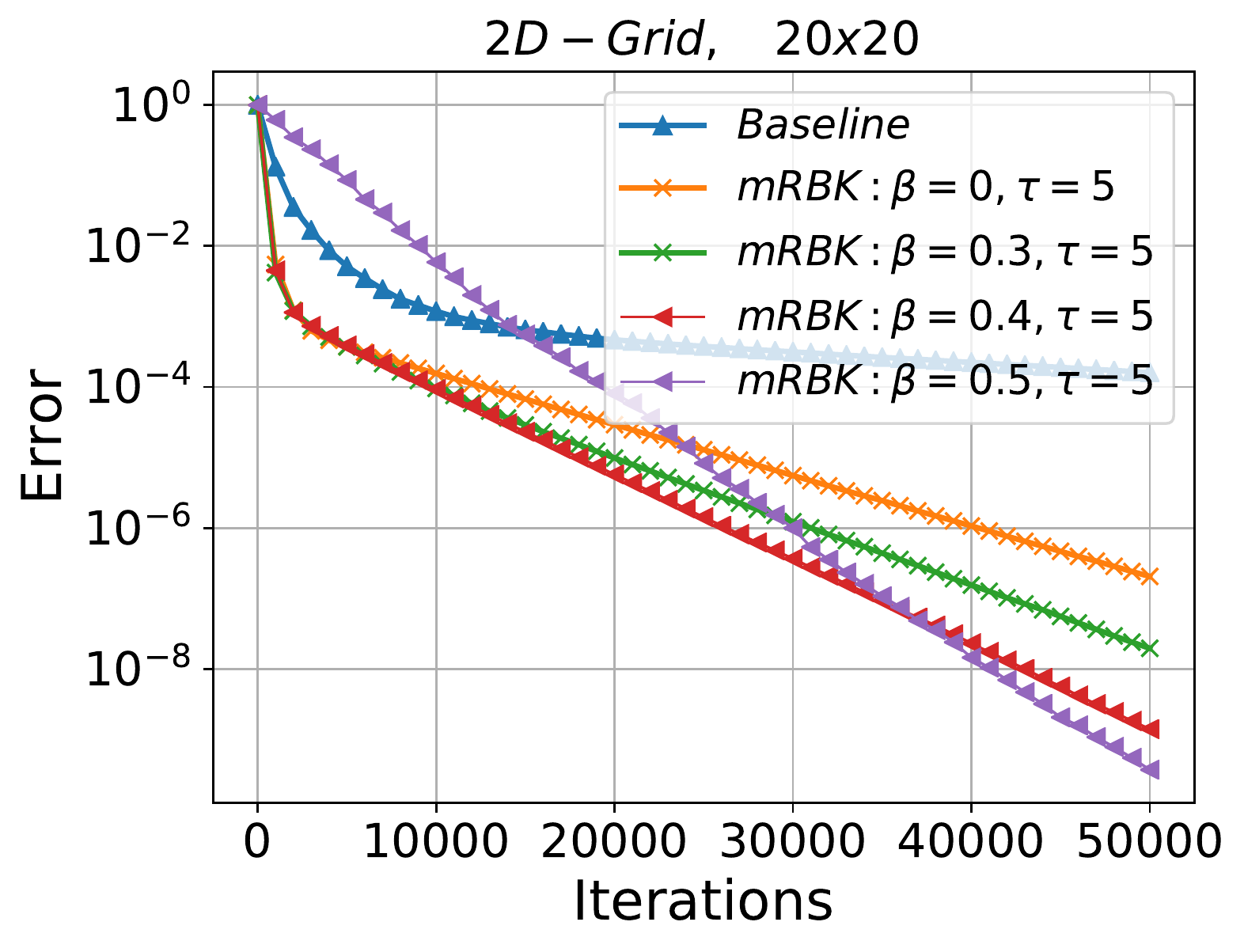}
\end{subfigure}\\
\begin{subfigure}{.25\textwidth}
  \centering
  \includegraphics[width=1\linewidth]{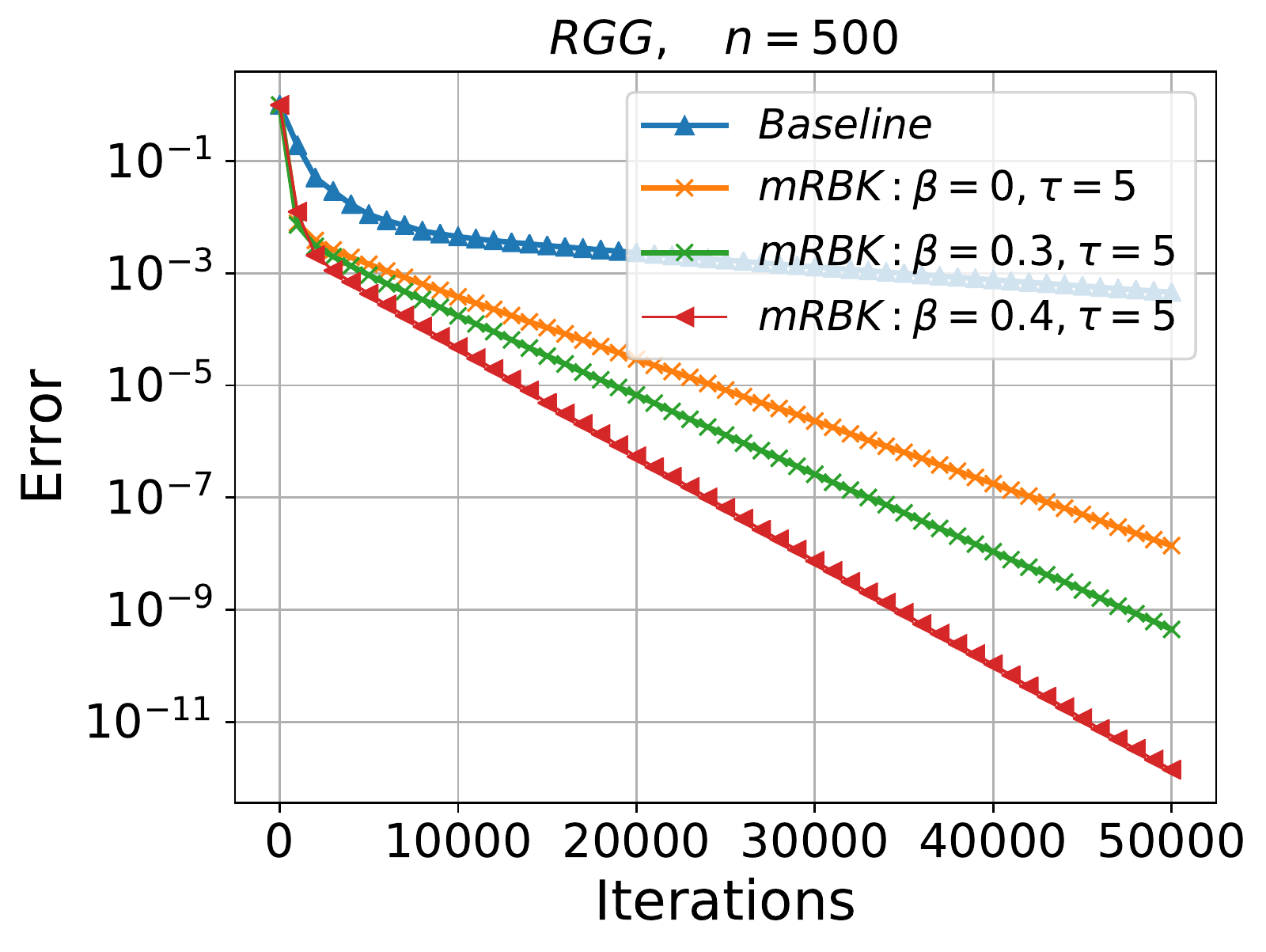}
\end{subfigure}%
\begin{subfigure}{.25\textwidth}
  \centering
  \includegraphics[width=1\linewidth]{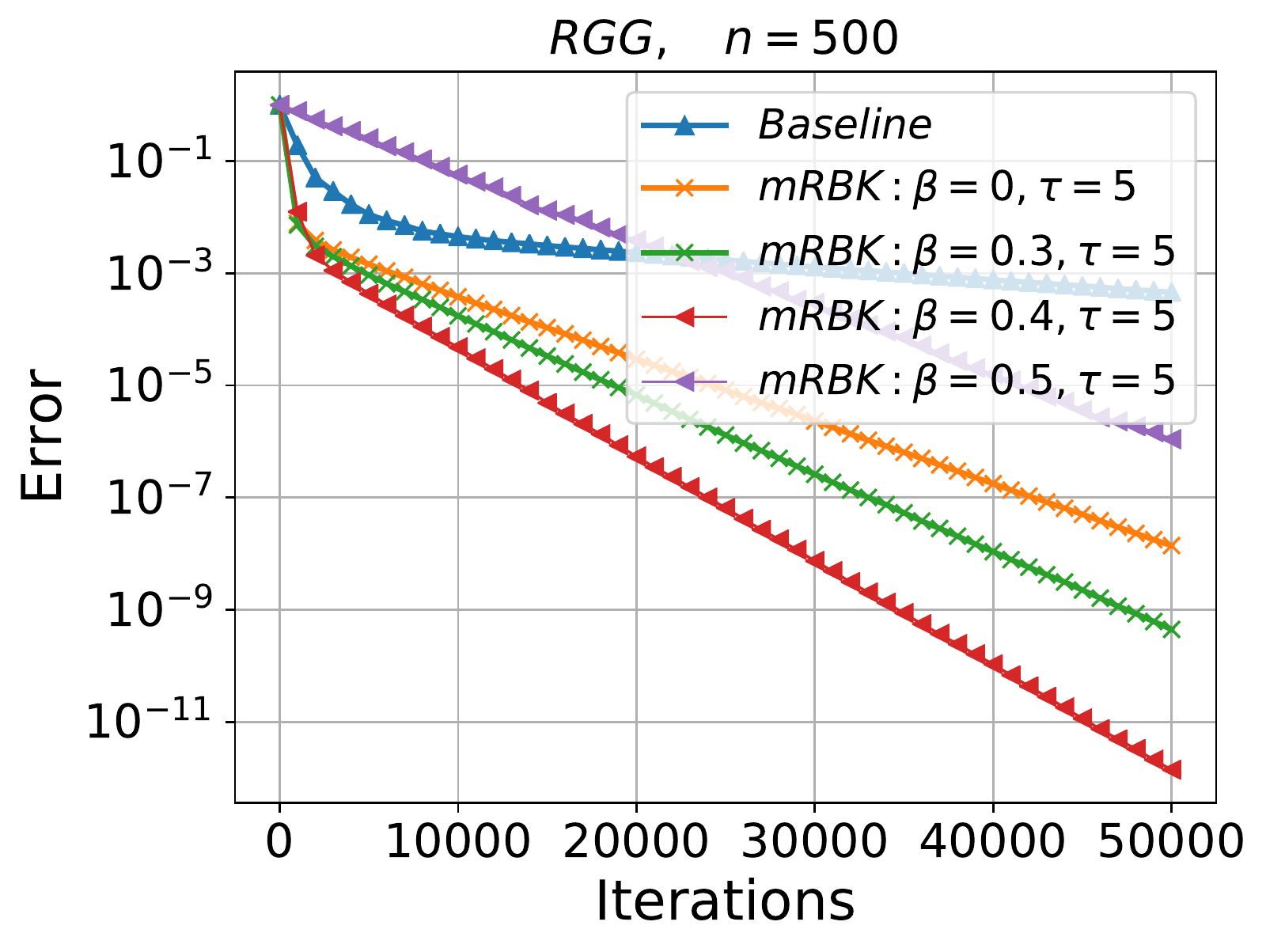}
\end{subfigure}\\
\begin{subfigure}{.25\textwidth}
  \centering
  \includegraphics[width=1\linewidth]{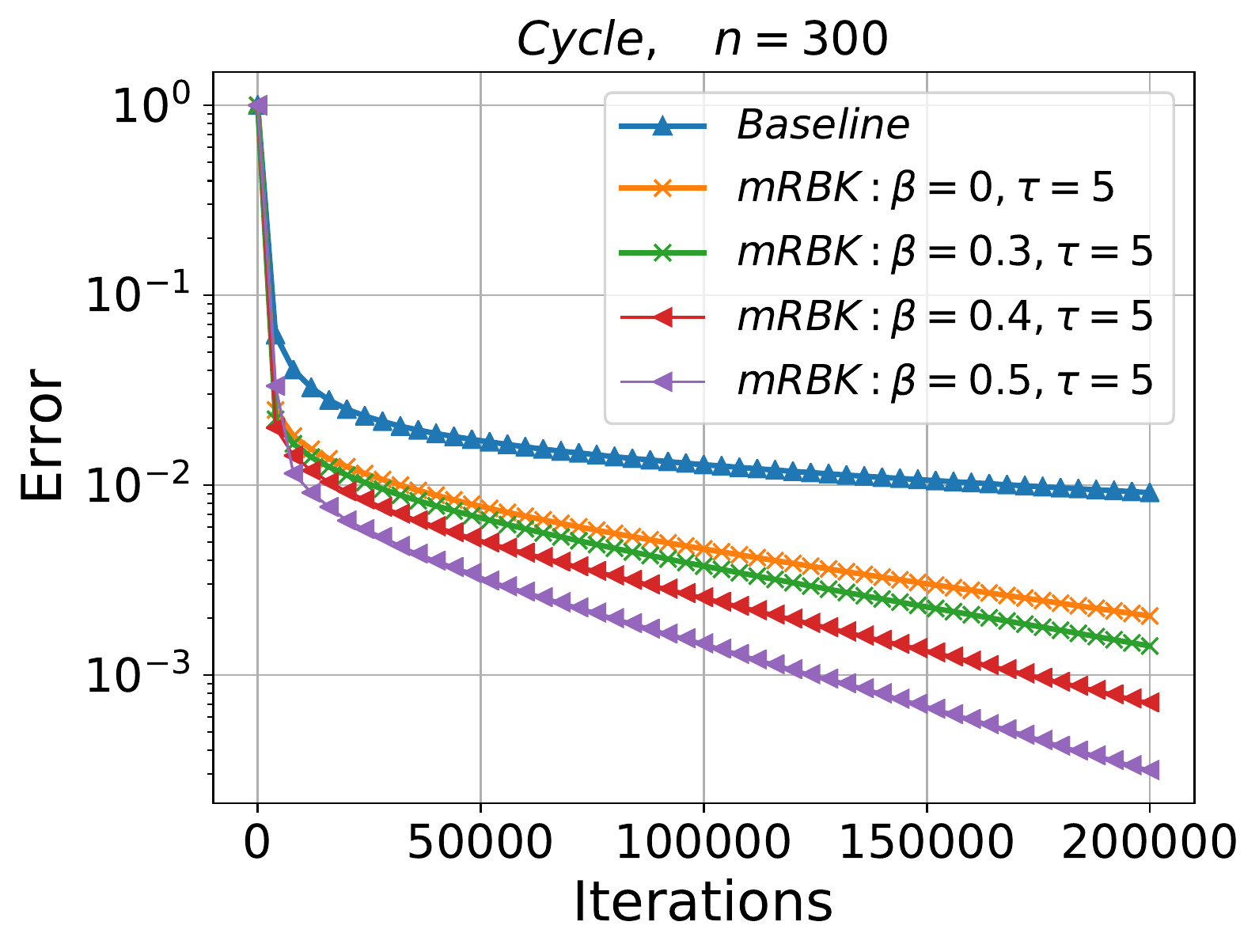}
\end{subfigure}%
\begin{subfigure}{.25\textwidth}
  \centering
  \includegraphics[width=1\linewidth]{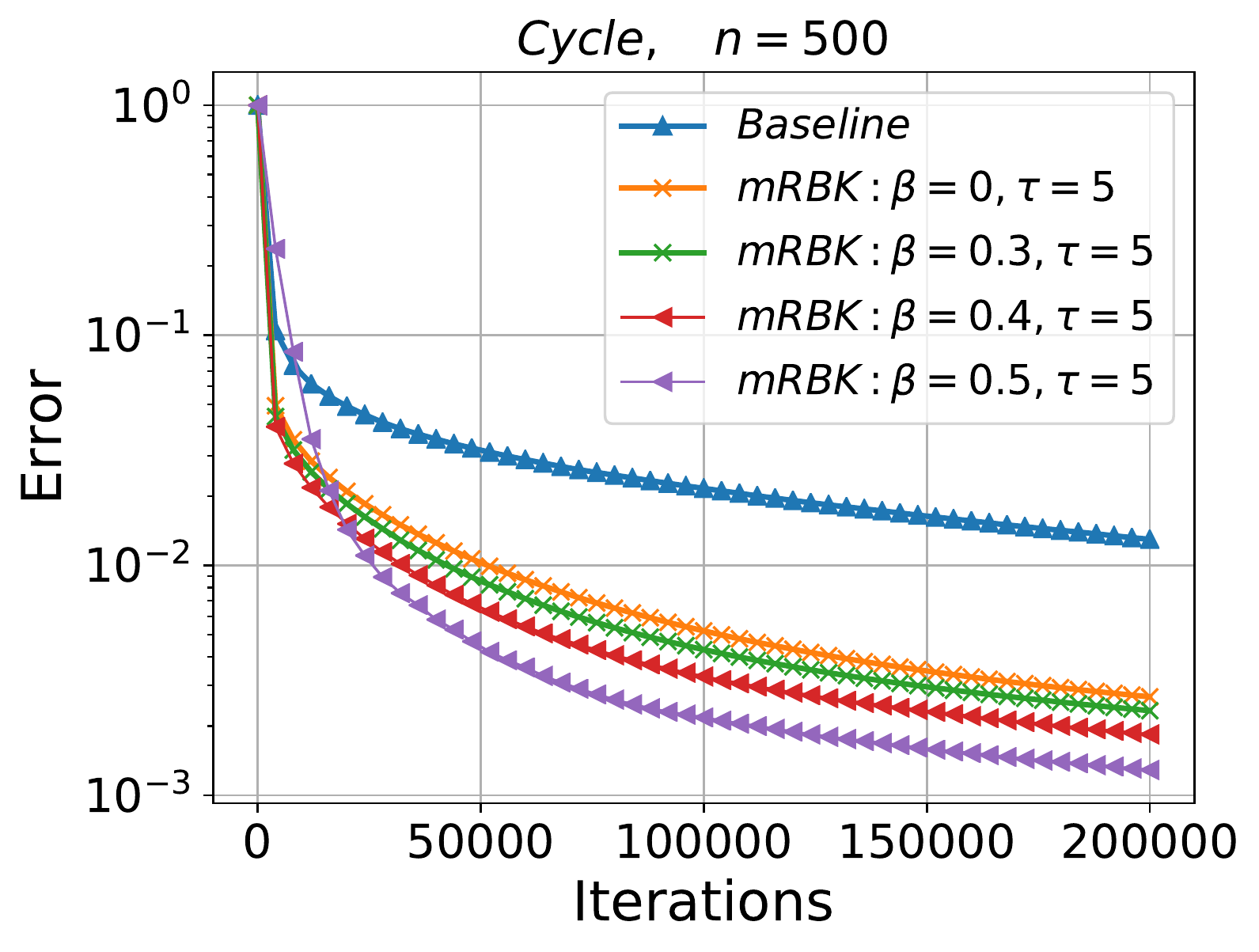}
  \label{fig:sub2}
\end{subfigure}
\caption{\footnotesize Comparison of mRBK with its no momentum variant RBK ($\beta=0$) proposed in \cite{LoizouRichtarik}.  The stepsize for all methods is $\omega=1$ and the block size is $\tau=5$. The baseline method in the plots denotes the simple randomized pairwise gossip algorithm (block $\tau=1$) and is plotted to highlight the benefits of having larger block sizes. The $n$ in the title of each plot indicates the number of nodes. For the grid graph this is $n \times n$.}
\label{RBKfigures}
\end{figure}

 \section{Conclusion and Future research}
 \label{conclusion}
In this paper we present new accelerated randomized gossip algorithms using tools from numerical linear algebra and the area of randomized Kaczmarz methods for solving linear systems. In particular, using recently developed results on the stochastic reformulation of consistent linear systems we explain how stochastic heavy ball method for solving a specific quadratic stochastic optimization problem can be interpreted as gossip algorithm. To the best of our knowledge, it is the first time that such protocols are presented for average consensus problem.  We believe that this work opens up many possible future venues for research. For example, using other Kaczmarz-type methods to solve particular linear systems we can obtain novel distributed protocols for average consensus.  In addition, we believe that the gossip protocols presented in this work can be extended to the more general setting of distributed optimization where the goal is to minimize the average of convex functions $(1/n) \sum_{i=1}^n f_i(x)$ in a distributed fashion. 

\bibliographystyle{plain}
\bibliography{SHBGossip}


\end{document}